\documentclass[11pt, twoside, leqno]{article}
\usepackage{amssymb}
\usepackage{amsmath}
\usepackage{amsthm}
\usepackage{color}
\usepackage{mathrsfs}

\usepackage{indentfirst}
\usepackage{graphicx}
\usepackage{txfonts}

\allowdisplaybreaks

\def\R2{{\mathbb R}^2}

\def\ls{\lesssim}
\def\gs{\gtrsim}

\def\XXint#1#2#3{{\setbox0=\hbox{$#1{#2#3}{\int}$ }
\vcenter{\hbox{$#2#3$ }}\kern-.6\wd0}}

\def\({\left(}
\def \){ \right)}

\newtheorem{theorem}{Theorem}[section]
\newtheorem{lemma}[theorem]{Lemma}

\newtheorem{proposition}[theorem]{Proposition}
\newtheorem{example}[theorem]{Example}
\theoremstyle{definition}
\newtheorem{remark}[theorem]{Remark}

\renewcommand{\appendix}{\par
   \setcounter{section}{0}%
   \setcounter{subsection}{0}%
   \setcounter{subsubsection}{0}%
   \gdef\thesection{\@Alph\c@section}%
   \gdef\thesubsection{\@Alph\c@section.\@arabic\c@subsection}%
   \gdef\theHsection{\@Alph\c@section.}%
   \gdef\theHsubsection{\@Alph\c@section.\@arabic\c@subsection}%
   \csname appendixmore\endcsname
 }

\numberwithin{equation}{section}

\begin{document}

\arraycolsep=1pt

\title{\bf\Large The Boundedness of the Bilinear Fractional Integrals along Curves
\footnotetext{\hspace{-0.35cm} 2020 {\it
Mathematics Subject Classification}. Primary 42B15; Secondary 42A85.
\endgraf {\it Key words and phrases.} Bilinear fractional integral, Convolution type operator, Bilinear interpolation, Restricted weak type estimate.
\endgraf Junfeng Li is supported by the National Natural Science Foundation of China (No.~12471090). Haixia Yu is supported by the National Natural Science Foundation of China (Nos.~12201378, 12471093), the Guangdong Basic and Applied Basic Research Foundation (Nos.~2023A1515010635, 2024A1515010468) and the Shantou University SRFT.}}
\author{Junfeng Li, Haixia Yu\footnote{Corresponding author.} ~and Minqun Zhao}

\date{}

\maketitle

\vspace{-0.7cm}

\begin{abstract}
 In this paper, for general curves $(t,\gamma(t))$ satisfying some suitable curvature conditions, we obtain some $L^p(\mathbb{R})\times L^q(\mathbb{R}) \rightarrow L^r(\mathbb{R})$ estimates for the bilinear fractional integrals $H_{\alpha,\gamma}$ along the curves $(t,\gamma(t))$, where
 $$H_{\alpha,\gamma}(f,g)(x):=\int_{0}^{\infty}f(x-t)g(x-\gamma(t))\,\frac{\textrm{d}t}{t^{1-\alpha}}$$
 and $\alpha\in (0,1)$. At the same time, we also establish an almost sharp Hardy-Littlewood-Sobolev inequality, i.e., the $L^p(\mathbb{R})\rightarrow L^q(\mathbb{R})$ estimate, for the fractional integral operators $I_{\alpha,\gamma}$ along the curves $(t,\gamma(t))$, where
$$I_{\alpha,\gamma}f(x):=\int_{0}^{\infty}\left|f(x-\gamma(t))\right|\,\frac{\textrm{d}t}{t^{1-\alpha}}.$$
\end{abstract}

\section{Introduction}

The purpose of this paper is to prove some $L^p(\mathbb{R})\times L^q(\mathbb{R}) \rightarrow L^r(\mathbb{R})$ estimates for bilinear operators of fractional integral type. The classical example of these operators is the bilinear fractional integral
\begin{align}\label{eq:1.00}
 H_{\alpha}(f,g)(x):=\int_{\mathbb{R}^n}f(x-t)g(x+t)\,\frac{\textrm{d}t}{|t|^{n-\alpha}}
 \end{align}
with $0<\alpha<n$, and the $L^p(\mathbb{R}^n)\times L^q(\mathbb{R}^n) \rightarrow L^r(\mathbb{R}^n)$ boundedness of $H_{\alpha}$ can be founded in Kenig and Stein \cite{KStein}, Grafakos and Kalton \cite{Grafk}, and also Grafakos \cite{Graf92}, where $(\frac{1}{p},\frac{1}{q},\frac{1}{r})$ lies in the open convex hull of the pentagon with vertices $(\frac{\alpha}{n}, 0, 0)$, $(1, 0, 1-\frac{\alpha}{n})$, $(0, \frac{\alpha}{n}, 0)$, $(0, 1, 1-\frac{\alpha}{n})$ and $(1, 1, \frac{2n-\alpha}{n})$. Furthermore, they also showed that a restricted weak type estimate holds at each vertex of the pentagon. On the other hand, by homogeneity, we note that $H_{\alpha}(f,g)$ maps $L^p(\mathbb{R}^n)\times L^q(\mathbb{R}^n)$ to $ L^r(\mathbb{R}^n)$ only when $$\frac{1}{r}=\frac{1}{p}+\frac{1}{q}-\frac{\alpha}{n}.$$

More precisely, in the present paper, we focus our attention on the bilinear fractional integral $H_{\alpha,\gamma}$ along a curve $(t,\gamma(t))$, which has the form
\begin{align}\label{eq:1.0}
 H_{\alpha,\gamma}(f,g)(x):=\int_{0}^{\infty}f(x-t)g(x-\gamma(t))\,\frac{\textrm{d}t}{t^{1-\alpha}},
 \end{align}
where $0<\alpha<1$. The first author of this paper and Liu considered the $L^p(\mathbb{R})\times L^q(\mathbb{R}) \rightarrow L^r(\mathbb{R})$ estimates for $H_{\alpha,\gamma}$ under the case of $\gamma(t):=t^\beta$ in \cite{LLiu}. Then, it is natural to characterize the indices $(\frac{1}{p}, \frac{1}{q}, \frac{1}{r})$ for which $H_{\alpha,\gamma}$ maps $L^p(\mathbb{R})\times L^q(\mathbb{R})$ to $L^r(\mathbb{R})$ for more general curves. We now state the needed conditions on the curve $(t,\gamma(t))$ and some related definitions.

\textbf{Hypothesis on curves (H.).} We assume that $\gamma\in C^2(\mathbb{R}^+)$ is monotonic on $\mathbb{R}^+$, $\lim\limits_{t\rightarrow 0^+}\gamma(t)=\lim\limits_{t\rightarrow 0^+}\gamma'(0)=0$, and satisfy $\omega_{\infty,1}\geq \omega_{0,2}$, $\omega_{0,1}>1$.
Furthermore, there exist positive constants $C^{(1)}_{1}$, $C^{(2)}_{1}$ and $C^{(1)}_2$ such that
\begin{align}\label{eq:1.1}
 \left|\frac{t\gamma'(t)}{\gamma(t)}\right|\leq C^{(1)}_2~~~\textrm{and}~~~\left|\frac{t^j\gamma^{(j)}(t)}{\gamma(t)}\right|\geq C^{(j)}_{1}~~\textrm{with}~~j=1,2.
\end{align}
In particular, we note that \textbf{(H.)} is essentially applicable to a polynomial with no linear term and constant term. Here and hereafter, we define
\begin{align*}
\omega_{0,1}:=\liminf\limits_{t\rightarrow0^{+}}{\frac{\ln{|\gamma(t)|}}{\ln t}},~~~~~&\omega_{0,2}:=\limsup\limits_{t\rightarrow0^{+}}{\frac{\ln{|\gamma(t)|}}{\ln t}},\\
\omega_{\infty,1}:=\liminf\limits_{t\rightarrow{\infty}}{\frac{\ln{|\gamma(t)|}}{\ln t}},~~~~~&\omega_{\infty,2}:=\limsup\limits_{t\rightarrow{\infty}}{\frac{\ln{|\gamma(t)|}}{\ln t}}.
\end{align*}

Next, we will state two main results of this paper. The first one is as follows.

\begin{theorem}\label{maintheorem}
Assume that $\gamma$ satisfies {\bf(H.)} and $\frac{\omega_{\infty,2}-\omega_{\infty,1}}{\omega_{\infty,2}-1}<\frac{\alpha}{1-\alpha}<\frac{\omega_{\infty,2}}{\omega_{\infty,2}-\omega_{\infty,1}+1}$. Then, there exists a positive constant $C$, such that
$$\left\|H_{\alpha,\gamma}(f,g)\right\|_{L^{r}(\mathbb{R})}\leq C \|f\|_{L^{p}(\mathbb{R})}\|g\|_{L^{q}(\mathbb{R})}$$
holds for all $f\in L^{p}(\mathbb{R})$ and $g\in L^{q}(\mathbb{R})$, where $(\frac{1}{p},\frac{1}{q},\frac{1}{r})$ lies in the open convex hull of the following points:
\begin{align*}
&A\left(\alpha,0,0\right),&&B\left(1,0,1-\alpha\right),\\
&C_1\left(0,\frac{\alpha}{\omega_{\infty,1}},0\right),&&C_2\left(0,\frac{\alpha}{\omega_{0,2}},0\right),\\
&D_1\left(0,1,1-\frac{\alpha}{\omega_{\infty,1}}\right),&&D_2\left(0,1,1-\frac{\alpha}{\omega_{0,2}}\right),\\
&E\left(1, \frac{\omega_{\infty,2}-\alpha}{2\omega_{\infty,2}-\omega_{\infty,1} }, \frac{2(\omega_{\infty,2}-\alpha)}{2\omega_{\infty,2}-\omega_{\infty,1} }\right),&&F\left(\frac{\omega_{\infty,2}-\alpha}{2\omega_{\infty,2}-\omega_{\infty,1} }, 1, \frac{2(\omega_{\infty,2}-\alpha)}{2\omega_{\infty,2}-\omega_{\infty,1} }\right).
\end{align*}
\end{theorem}

The second main result of this paper, i.e., Theorem \ref{maintheorem2}, will study the analogue of Theorem \ref{maintheorem} for the case of $\frac{\alpha}{1-\alpha}\geq \omega_{\infty,1}$. We note that Theorem \ref{maintheorem2} may also be viewed as a necessary supplement corresponding to Theorem \ref{maintheorem}, since the condition $\frac{\omega_{\infty,2}-\omega_{\infty,1}}{\omega_{\infty,2}-1}<\frac{\alpha}{1-\alpha}<\frac{\omega_{\infty,2}}{\omega_{\infty,2}-\omega_{\infty,1}+1}$ in Theorem \ref{maintheorem} is equivalent to $\frac{\alpha}{1-\alpha}<\omega_{\infty,1}$ when $\omega_{\infty,1}=\omega_{\infty,2}$.

\begin{theorem}\label{maintheorem2}
Assume that $\gamma$ satisfies {\bf(H.)} and $\frac{\alpha}{1-\alpha}\geq \omega_{\infty,1}$. Then, there exists a positive constant $C$, such that
$$\left\|H_{\alpha,\gamma}(f,g)\right\|_{L^{r}(\mathbb{R})}\leq C \|f\|_{L^{p}(\mathbb{R})}\|g\|_{L^{q}(\mathbb{R})}$$
holds for all $f\in L^{p}(\mathbb{R})$ and $g\in L^{q}(\mathbb{R})$, where $(\frac{1}{p},\frac{1}{q},\frac{1}{r})$ lies in the open convex hull of the following points:
\begin{align*}
&~~~A\left(\alpha,0,0\right),&&B\left(1,0,1-\alpha\right),\\
&~~~C_1\left(0,\frac{\alpha}{\omega_{\infty,1}},0\right),&&C_2\left(0,\frac{\alpha}{\omega_{0,2}},0\right),\\
&~~~D_1\left(0,1,1-\frac{\alpha}{\omega_{\infty,1}}\right),&&D_2\left(0,1,1-\frac{\alpha}{\omega_{0,2}}\right),\\
&\begin{cases}
G_1\left(1, \frac{1-\alpha}{2-\omega_{\infty,1}}, \frac{2(1-\alpha)}{2-\omega_{\infty,1}}\right),&\omega_{\infty,1}\leq2\alpha;\\
G^1_2\left(1, \frac{\alpha}{\omega_{\infty,1}}, 1\right),~G^2_2\left(1, \frac{1}{2}, 1\right),&2\alpha<\omega_{\infty,1},
\end{cases},&&H\left(\frac{\omega_{\infty,1}-\alpha}{\omega_{\infty,1}}, 1, \frac{2(\omega_{\infty,1}-\alpha)}{\omega_{\infty,1}}\right).
\end{align*}
\end{theorem}

\begin{remark}
Let us express the regions of $(\frac{1}{p}, \frac{1}{q}, \frac{1}{r})$ in Theorems \ref{maintheorem} and \ref{maintheorem2} as in Figures \ref{Figure:1} and \ref{Figure:2}, respectively.
\begin{figure}[htbp]
  \centering
  \includegraphics[width=4.5in]{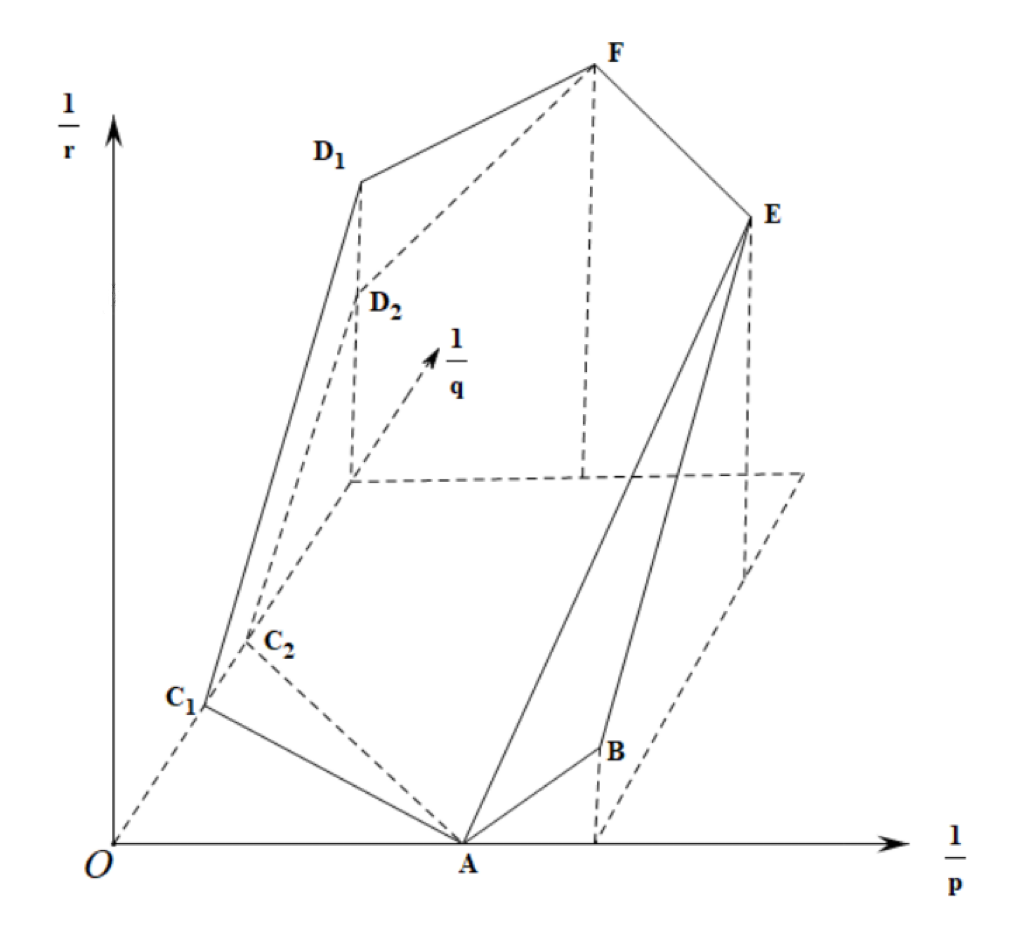}\\
  \caption{Regions of $(\frac{1}{p}, \frac{1}{q}, \frac{1}{r})$ in Theorem \ref{maintheorem}.}\label{Figure:1}
\end{figure}
\begin{figure}[htbp]
  \centering
  \includegraphics[width=5.68in]{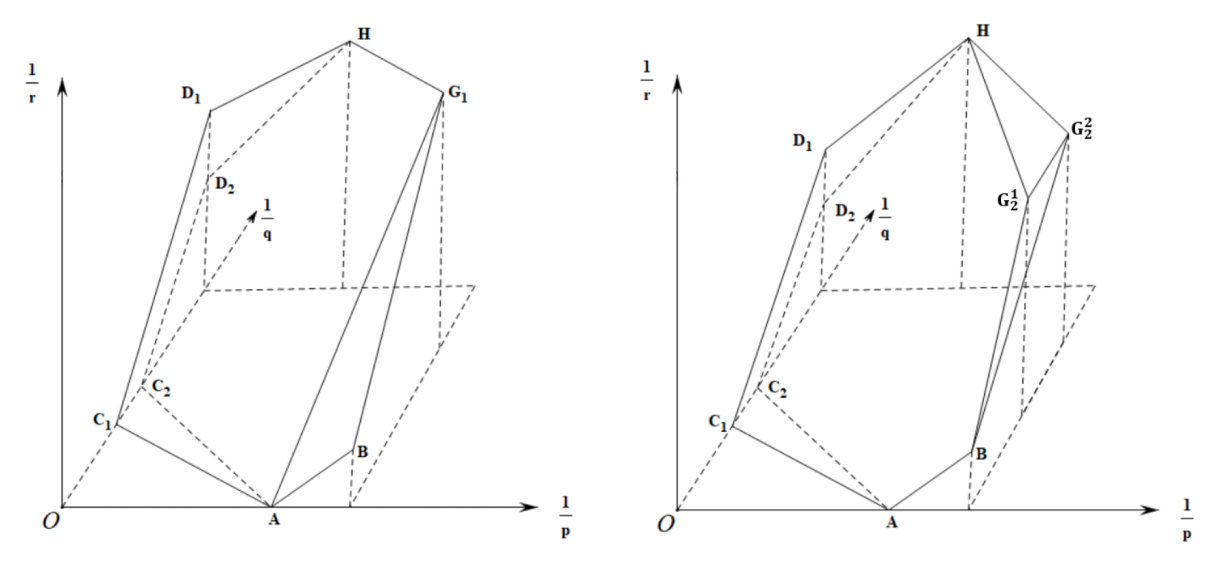}\\
  \caption{Regions of $(\frac{1}{p}, \frac{1}{q}, \frac{1}{r})$ in Theorem \ref{maintheorem2} with $\omega_{\infty,1}\leq2\alpha$ and $2\alpha<\omega_{\infty,1}$, respectively.}\label{Figure:2}
\end{figure}
\end{remark}

\begin{example}
Let us list some examples of curves.
\begin{enumerate}
\item[\rm(1)] $\gamma_1(t):=t^\beta$ with $\beta\in(1,\infty)$. Then the regions of $(\frac{1}{p}, \frac{1}{q}, \frac{1}{r})$ in the $L^p(\mathbb{R})\times L^q(\mathbb{R}) \rightarrow L^r(\mathbb{R})$ estimates for the corresponding $H_{\alpha,\gamma_1}$ lies in the open convex hull of the following points:
\begin{align*}
&~~~A\left(\alpha,0,0\right),~~~~~~~~~~~~~~~~~~~~~~~~~~~~~B\left(1,0,1-\alpha\right),\\
&~~~C\left(0,\frac{\alpha}{\beta},0\right),~~~~~~~~~~~~~~~~~~~~~~~~~~~D\left(0,1,1-\frac{\alpha}{\beta}\right),\\
&\begin{cases}
 E\left(1, 1-\frac{\alpha}{\beta}, 2(1-\frac{\alpha}{\beta})\right),~~~~~~~~~~F\left(1-\frac{\alpha}{\beta}, 1, 2(1-\frac{\alpha}{\beta})\right),&  if~ \frac{\alpha}{1-\alpha}<\beta;\\
 G_1\left(1, \frac{1-\alpha}{2-\beta}, \frac{2(1-\alpha)}{2-\beta}\right),~~~~~~~~~~~~~~~H\left(1-\frac{\alpha}{\beta}, 1, 2(1-\frac{\alpha}{\beta})\right),& if~ \frac{\alpha}{1-\alpha} \geq \beta~and~\beta \leq 2\alpha;\\
G^1_2\left(1, \frac{\alpha}{\beta}, 1\right),~~G^2_2\left(1, \frac{1}{2}, 1\right),~~~H\left(1-\frac{\alpha}{\beta}, 1, 2(1-\frac{\alpha}{\beta})\right),& if~ \frac{\alpha}{1-\alpha} \geq \beta ~and~2\alpha<\beta.
\end{cases}
\end{align*}
\item[\rm(2)] $\gamma_2(t):=\sum\limits_{i=1}^{k} \beta_i t^{\alpha_i}$, where $\beta_i\in \mathbb{R}\backslash\{0\} $ for all $i=1,2,\cdots,k$ and $1<\alpha_1<\alpha_2<\cdots<\alpha_k$ with $k\in \mathbb{N}$. Then the regions of $(\frac{1}{p}, \frac{1}{q}, \frac{1}{r})$ in the $L^p(\mathbb{R})\times L^q(\mathbb{R}) \rightarrow L^r(\mathbb{R})$ estimates for the corresponding $H_{\alpha,\gamma_2}$ lies in the open convex hull of the following points:
 \begin{align*}
&~~~A\left(\alpha,0,0\right),~~~~~~~~~~~~~~~~~~~~~~~~~~~~~~~B\left(1,0,1-\alpha\right),\\
&~~~C_1\left(0,\frac{\alpha}{\alpha_k},0\right),~~~~~~~~~~~~~~~~~~~~~~~~~~C_2\left(0,\frac{\alpha}{\alpha_1},0\right),\\
&~~~D_1\left(0,1,1-\frac{\alpha}{\alpha_k}\right),~~~~~~~~~~~~~~~~~~~D_2\left(0,1,1-\frac{\alpha}{\alpha_1}\right),\\
&\begin{cases}
 E\left(1, 1-\frac{\alpha}{\alpha_k}, 2(1-\frac{\alpha}{\alpha_k})\right),~~~~~~~~~F\left(1-\frac{\alpha}{\alpha_k}, 1, 2(1-\frac{\alpha}{\alpha_k})\right),& if~ \frac{\alpha}{1-\alpha}<\alpha_k;\\
 G_1\left(1, \frac{1-\alpha}{2-\alpha_k}, \frac{2(1-\alpha)}{2-\alpha_k}\right),~~~~~~~~~~~~~~~H\left(1-\frac{\alpha}{\alpha_k}, 1, 2(1-\frac{\alpha}{\alpha_k})\right),&  if~ \frac{\alpha}{1-\alpha} \geq \alpha_k~and~\alpha_k \leq 2\alpha;\\
G^1_2\left(1, \frac{\alpha}{\alpha_k}, 1\right),~~G^2_2\left(1, \frac{1}{2}, 1\right),~~~ H\left(1-\frac{\alpha}{\alpha_k}, 1, 2(1-\frac{\alpha}{\alpha_k})\right),& if~ \frac{\alpha}{1-\alpha} \geq \alpha_k ~and~2\alpha<\alpha_k.
\end{cases}
\end{align*}
\item[\rm(3)] $\gamma_3(t):=t^\beta\log(1+t)$, or $t^\beta \arctan t$, or $e^{\beta t}-\beta t-1$, where $\beta>0$. We note that all of these $\gamma_3$  satisfy $\gamma_3(0)=\gamma'_3(0)=\cdots=\gamma^{(d-1)}_3(0)=0$ and $\gamma^{(d)}_3(0)\neq0$, where $d\geq 2$ and $d\in\mathbb{N}$, which is called as finite type $d$ at $0$ in Iosevich \cite{Iose}. The regions of $(\frac{1}{p}, \frac{1}{q}, \frac{1}{r})$ in the $L^p(\mathbb{R})\times L^q(\mathbb{R}) \rightarrow L^r(\mathbb{R})$ estimates for the corresponding $H_{\alpha,\gamma_3}$ can be obtained as in above, we omit the details.
\end{enumerate}
\end{example}

The first motivation of our problem is the bilinear Hilbert transform $H_{\gamma}(f,g)$ along a curve $(t,\gamma(t))$, which is defined by
$$H_{\gamma}(f,g)(x):=\mathrm{p.\,v.}\int_{-\infty}^{\infty}f(x-t)g(x-\gamma(t))\,\frac{\textrm{d}t}{t}.$$
If $\gamma(t):=-t$, then this operator turns to be the standard bilinear Hilbert transform. As a landmark work, Lacey and Thiele \cite{LT1,LT2} set up its $L^p(\mathbb{R})\times L^q(\mathbb{R})\rightarrow L^r(\mathbb{R})$ boundedness for all $\frac{1}{p}+\frac{1}{q}=\frac{1}{r}$, $p, q>1$, and $r>\frac{2}{3}$. This boundedness confirms a conjecture raised by Calderon. The restriction $r>\frac{2}{3}$ follows from the boundedness of a model operator they used. So far, to extend the region to $\frac{1}{2} < r \leq \frac{2}{3}$ is still a challenging question. Related works can be found in \cite{GLi04,LiXC,LiLi22} and references therein.

On the other hand, the $L^p(\mathbb{R})\times L^q(\mathbb{R})\rightarrow L^r(\mathbb{R})$ boundedness of $H_{\gamma}(f,g)$ with some general curves $\gamma$ are of great interest to us. We start with a special case $\gamma(t):=t^d$, where $d\in \mathbb{N}$ and $d>1$, Li \cite{LiX} obtained the $L^2(\mathbb{R})\times L^2(\mathbb{R})\rightarrow L^1(\mathbb{R})$ boundedness for this operator. For $\gamma(t):=\textrm{P}(t)$, a polynomial with no linear term and constant term. Li and Xiao \cite{LX} set up the $L^p(\mathbb{R})\times L^q(\mathbb{R})\rightarrow L^r(\mathbb{R})$ boundedness of $H_{\gamma}(f,g)$ with $\frac{1}{p}+\frac{1}{q}=\frac{1}{r}$, $p>1$, $q>1$ and $r>\frac{d-1}{d}$. At the same time, they showed that $r>\frac{d-1}{d}$ is sharp up to the end point. Moreover, by replacing $\gamma$ with a homogeneous curve $t^d$, where $d\in \mathbb{N}$ and $d>1$, they also showed that the range of $r$ can be extended to $r>\frac{1}{2}$. Lie \cite{Lie1} considered a curve class $\mathcal{NF}^C$ and obtained the $L^2(\mathbb{R})\times L^2(\mathbb{R})\rightarrow L^1(\mathbb{R})$ boundedness of $H_{\gamma}(f,g)$ for $\gamma\in \mathcal{NF}^C$, and it was extended to the $L^p(\mathbb{R})\times L^q(\mathbb{R})\rightarrow L^r(\mathbb{R})$ boundedness with $r\geq 1$ in \cite{Lie2}. More recently, Guo and Xiao \cite{GX} also introduced a set $\textbf{F}(-1,1)$ and obtained the $L^2(\mathbb{R})\times L^2(\mathbb{R})\rightarrow L^1(\mathbb{R})$ boundedness of $H_{\gamma}(f,g)$ for $\gamma\in \textbf{F}(-1,1)$. Li and Yu \cite{LY3} obtained the $L^p(\mathbb{R})\times L^q(\mathbb{R})\rightarrow L^r(\mathbb{R})$ boundedness of $H_{\gamma}(f,g)$ with $r>\frac{1}{2}$ for some curves that are much easier to be
verified than $\mathcal{NF}^C$ and $\textbf{F}(-1,1)$. For some other related works, we refer the reader to \cite{CDR21,CG24,LiLi23}.

The second motivation of our problem comes from the Hilbert transform $T_{\gamma}$ along a curve $(t,\gamma(t))$, where
$$T_{\gamma}f(x_1,x_2):=\mathrm{p.\,v.}\int_{-\infty}^{\infty}f(x_1-t,x_2-\gamma(t))
\,\frac{\textrm{d}t}{t}.$$
This and related families have been extensively studied. The $L^p(\mathbb{R}^2)$ boundedness of $T_{\gamma}$ with $\gamma(t):=t^2$ arisen naturally within the study of the constant coefficient parabolic  differential operators initiated by Jones \cite{J}. Later, it was extended to more general classes of curves; see for instance \cite{SW,NVWW,CCCD,CNVWW,CCVWW,CVWW} and the references therein. There are some results which extend the aforementioned results to variable coefficient settings; see, for example, \cite{CWW,BJ,GHLR,DGTZ,LiYu21,LiYu22,LYu}.

On the other hand, we can define the fractional integral $ T_{\alpha,\gamma}$ along a curve $(t,\gamma(t))$ as
\begin{align*}
 T_{\alpha,\gamma}f(x_1,x_2):=\int_{0}^{\infty}f(x_1-t,x_2-\gamma(t))\,\frac{\textrm{d}t}{t^{1-\alpha}},
 \end{align*}
where $0<\alpha<1$. The study of the boundedness properties of $T_{\alpha,\gamma}$ first appeared in the work of Ricci and Stein \cite{RStein}. The discrete operator modeled on the fractional integral was proved by Stein and Wainger \cite{SteW00}. Moreover, Pierce \cite{Pierce} considered the discrete analog of a fractional integral operator on the Heisenberg group. For higher dimensional case, we refer the reader to Secco \cite{Secco} and Gressman \cite{Gressman}. For more related works, see \cite{Grafa93,SeWain}.

We now want to mention several ingredients in the proofs of our Theorems \ref{maintheorem} and \ref{maintheorem2}. The basic idea in our proofs is the interpolation argument. To treat the case when $p=\infty$ or $q=\infty$, i.e., the points: $A, B, C_1, C_2, D_1$ and $D_2$, we will obtain some restricted weak type estimates for $H_{\alpha,\gamma}$. The main difficulty is to prove the Hardy-Littlewood-Sobolev inequality, i.e., the $L^p(\mathbb{R})\rightarrow L^q(\mathbb{R})$ estimate for the corresponding fractional integral operator $I_{\alpha,\gamma}$. We here use a property of $\gamma$ and find the desired estimates depend on the indices $\omega_{0,1}, \omega_{0,2}, \omega_{\infty,1}$ and $\omega_{\infty,2}$ in \textbf{(H.)}, see Lemma \ref{lemma 2.2}. For the case when $p=1$ or $q=1$, i.e., the points: $E, F, G_1, G_2^1, G_2^2$ and $H$, we shall consider the non-critical case and critical case in establishing a uniform estimate for $H_{\alpha,\gamma,j}$, which comes from a decomposition of $H_{\alpha,\gamma}$. Moreover, to make $\frac{1}{r}$ or $\frac{1}{q}$ tends to maximal value when $p=1$, we need some complicated calculations by considering all possibilities. Argue similarly when $q=1$, but we here need to bound \eqref{eq:5.a} by \eqref{eq:5.b} and \eqref{eq:5.000} respectively, which is the most technical parts of the proofs with $q=1$. The whole purpose of bounding \eqref{eq:5.a} is to find a certain balance between the following two cases $|E_\lambda|\geq2^{j_0}|\gamma'(2^{j_{0}-1})|$ and $|E_\lambda|<2^{j_0}|\gamma'(2^{j_{0}-1})|$. It is also worth noticing that the stationary point of $H_{\alpha,\gamma}$ will leads to a few technical difficulties, which does not occur in considering $H_{\alpha}$ defined in \eqref{eq:1.00}, see \cite{KStein,Grafk,Graf92}.

The structure of the paper is as follows:
\begin{enumerate}
  \item[$\ast$] In Section 2, we get some restricted weak type estimates for $H_{\alpha,\gamma}$ at $p=\infty$ or $q=\infty$ based on an estimate for the fractional integral operator $I_{\alpha,\gamma}$.

  \item[$\ast$] In Section 3, we devote to establish a uniform estimate for $H_{\alpha,\gamma,j}$ by considering two cases: non-critical case and critical case, see Propositions \ref{proposition 3.1} and \ref{proposition 3.2} respectively.

  \item[$\ast$] In Section 4, we obtain some $L^p(\mathbb{R})\times L^q(\mathbb{R}) \rightarrow L^{r,\infty}(\mathbb{R})$ estimates for $H_{\alpha,\gamma}$ with $p=1$. The estimate near the point $E$ in Theorem \ref{maintheorem} can be found in Proposition \ref{proposition 4.1} and the estimate near the points $G_1$, $G_2^1$ and $G_2^2$ in Theorem \ref{maintheorem2} can be found in Proposition \ref{proposition 4.2}, where $\frac{1}{q}$ or $\frac{1}{r}$ tends to maximal value in our proofs.

  \item[$\ast$] In Sections 5, we obtain some $L^p(\mathbb{R})\times L^q(\mathbb{R}) \rightarrow L^{r,\infty}(\mathbb{R})$ estimates for $H_{\alpha,\gamma}$ with $q=1$. Proposition \ref{proposition 5.1} shows the estimate near the point $F$ in Theorem \ref{maintheorem} and Proposition \ref{proposition 5.2} is devote to the estimate near the point $H$ in Theorem \ref{maintheorem2}. We note that $\frac{1}{p}$ or $\frac{1}{r}$ also tends to maximal value in our proofs.
\end{enumerate}

Throughout this paper, we always use $C$ to denote a positive constant, which is independent of the main parameters,
but it may vary from line to line. Moreover, we use $C_{\varepsilon}$ to denote a positive constant depending on the indicated
parameter $\varepsilon$. Furthermore, we will use the standard notation $a\ls b$ (or $a\gs b$) to mean that there exist a positive constant $C$ such that $a\le Cb$ (or $a\ge Cb$). The expression $a\approx b$ means $a\ls b$ and $b\ls a$. $\hat{f}$ means the Fourier transform of $f$. For any set $E$, we use $\chi_E$ to denote the {characteristic function} of $E$, and $|E|$ shall denote the measure of $E$. Let $\mathbb{R}^+:=(0,\infty)$, and we use $\infty$ to denote $+\infty$. For any $s\in \mathbb{R}$, we denote by $\lceil s\rceil$ the smallest integer greater than $s$.

\section{Some restricted weak type estimates for $H_{\alpha,\gamma}$ at $p=\infty$ or $q=\infty$}

In this section, the aim is to prove some restricted weak type estimates for $H_{\alpha,\gamma}$ at particula endpoints $p=\infty$ or $q=\infty$. We commence by considering the fractional integral operator $I_{\alpha,\gamma}$ along the curve $(t,\gamma(t))$, where
\begin{align}\label{eq:2.4}
I_{\alpha,\gamma}f(x):=\int_{0}^{\infty}\left|f(x-\gamma(t))\right|\,\frac{\textrm{d}t}{t^{1-\alpha}}
 \end{align}
with $0<\alpha<1$, which is a convolution type operator.

\subsection {An estimate for the fractional integral operator $I_{\alpha,\gamma}$}

Let us begin by introducing the following lemma.

\begin{lemma}\label{lemma 2.1}
Assume that $\gamma\in C^{1}(\mathbb{R}^+)$ is monotonic on $\mathbb{R}^+$, $\lim\limits_{t\rightarrow 0^+}\gamma(t)=0$ and satisfies that
\begin{align}\label{eq:2.1}
there~ exist~ positive ~constants~ C^{(1)}_1~ and~ C^{(1)}_2 ~such~ that~ C^{(1)}_1\leq\left|\frac{t\gamma'(t)}{\gamma(t)}\right|\leq C^{(1)}_2~ on~ \mathbb{R}^+.
\end{align}
Then, for any $\varepsilon>0$, there exist positive constants $t_{1,\varepsilon}\in(0,1)$ and $t_{2,\varepsilon}\in(1,\infty)$, such that
$$\begin{cases}
  |\gamma(t)|\in(t^{\omega_{0,2}+\varepsilon},t^{\omega_{0,1}-\varepsilon})~ &\textrm{for}~ \textrm{all} ~t\in(0,t_{1,\varepsilon}); \\
  |\gamma(t)|\in(t^{\omega_{\infty,1}-\varepsilon},t^{\omega_{\infty,2}+\varepsilon})~ &\textrm{for}~ \textrm{all} ~t\in(t_{2,\varepsilon},\infty),
 \end{cases}$$
where the definitions of $\omega_{0,1}$, $\omega_{0,2}$, $\omega_{\infty,1}$ and $\omega_{\infty,2}$ can be found in {\bf(H.)}.
\end{lemma}

\begin{proof}
We first show $\omega_{0,1}, \omega_{0,2}, \omega_{\infty,1}, \omega_{\infty,2}\in (0,\infty)$. Indeed, it follows from Liu, Song and Yu \cite[Remark 1.4]{LSY} that
\begin{align}\label{eq:2.2}
e^{\frac{C^{(1)}_{1}}{2}}\leq \frac{|\gamma(2t)|}{|\gamma(t)|}\leq e^{C^{(1)}_{2}}
  \end{align}
for all $t\in \mathbb{R}^+$, we then have $\lim\limits_{t\rightarrow \infty}|\gamma(t)|=\infty$. This, combined with $\lim\limits_{t\rightarrow 0^+}\gamma(t)=0$, by l'H\^{o}pital's Rule and \eqref{eq:2.1}, leads to
\begin{align*}
\frac{C^{(1)}_1}{2}\leq{\frac{\ln{|\gamma(t)|}}{\ln t}}\leq 2C^{(1)}_2
  \end{align*}
holds for all $t\in (0, A)\cup (A^{-1}, \infty)$, where $A$ is a positive constant. Furthermore, from the supremum and infimum principle, we can find that $\omega_{0,1}, \omega_{0,2}, \omega_{\infty,1}, \omega_{\infty,2}\in [\frac{C^{(1)}_1}{2}, 2C^{(1)}_2] \subset(0,\infty)$.

From the definitions of $\omega_{0,1}$ and $\omega_{0,2}$, we see that $$\omega_{0,1}=\sup\limits_{\tau\in (0, 1)}\inf\limits_{t\in (0, \tau)}\frac{\ln|\gamma(t)|}{\ln t} ~~\textrm{and} ~~\omega_{0,2}=\inf\limits_{\tau\in (0,1)}\sup\limits_{t\in (0,\tau)}\frac{\ln|\gamma(t)|}{\ln t}.$$ Then, for any $\varepsilon>0$, there exists a positive constant $t_{1,\varepsilon}\in(0,1)$ such that $$\inf\limits_{t\in (0, t_{1,\varepsilon})}\frac{\ln|\gamma(t)|}{\ln t}\in(\omega_{0,1}-\varepsilon, \omega_{0,1}]~~ \textrm{and} ~~\sup_{t\in (0, t_{1,\varepsilon})}\frac{\ln|\gamma(t)|}{\ln t}\in [\omega_{0,2}, \omega_{0,2}+\varepsilon),$$ which implies that $$\omega_{0,1}-\varepsilon<\frac{\ln|\gamma(t)|}{\ln t}<\omega_{0,2}+\varepsilon$$ for all $t\in (0, t_{1,\varepsilon})$. Therefore,
\begin{align*}
|\gamma(t)|=t^{\log_t {|\gamma(t)|}}=t^{\frac{\ln |\gamma(t)|}{\ln t}}\in(t^{\omega_{0,2}+\varepsilon},t^{\omega_{0,1}-\varepsilon})
\end{align*}
for all $t\in (0, t_{1,\varepsilon})$. Similarly, for the same $\varepsilon$, there exists a positive constant $t_{2,\varepsilon}\in(1,\infty)$ such that $$\omega_{\infty,1}-\varepsilon<\frac{\ln|\gamma(t)|}{\ln t}<\omega_{\infty,2}+\varepsilon$$ for all $t\in (t_{2,\varepsilon}, \infty)$, which leads to
\begin{align*}
|\gamma(t)|\in(t^{\omega_{\infty,1}-\varepsilon},t^{\omega_{\infty,2}+\varepsilon})
\end{align*}
for all $t\in (t_{2,\varepsilon}, \infty)$. Therefore, we finish the proof of Lemma \ref{lemma 2.1}.
\end{proof}

Based on this Lemma \ref{lemma 2.1} we are invited to establish an almost sharp $L^p(\mathbb{R})\rightarrow L^q(\mathbb{R})$ estimate for $I_{\alpha,\gamma}$ defined in \eqref{eq:2.4}.

\begin{lemma}\label{lemma 2.2}
Let $\gamma $ be the same as in Lemma \ref{lemma 2.1} with $\omega_{0,2}\leq\omega_{\infty,1}$, and $1<p<q<\infty$ satisfy
\begin{align*}
\frac{1}{p}-\frac{1}{q}=\frac{\alpha}{M} ~~~~~~with~ M\in (\omega_{0,2}, \omega_{\infty,1}).
\end{align*}
Then\footnote{If $\omega_{0,2}=\omega_{\infty,1}$, an alternative result can be found in the following Remark \ref{remark 2.3}.}, there exists a positive constant $C$, such that for all $f\in L^p(\mathbb{R})$,
\begin{align}\label{eq:2.5}
\left\|I_{\alpha,\gamma}f\right\|_{L^{q}(\mathbb{R})} \leq C \|f\|_{L^{p}(\mathbb{R})}.
\end{align}
On the other hand, suppose that the estimate \eqref{eq:2.5} holds. Then $p$ and $q$ satisfy
\begin{align*}
\frac{1}{p}-\frac{1}{q}=\frac{\alpha}{N}~~~~~~with~N\in [\omega_{0,1}, \omega_{\infty,2}].
\end{align*}
\end{lemma}

\begin{proof} For any $\varepsilon>0$, using the positive constants $t_{1,\varepsilon}\in(0,1)$ and $t_{2,\varepsilon}\in(1,\infty)$ in Lemma \ref{lemma 2.1}, we decompose
\begin{align*}
I_{\alpha,\gamma}f(x)=\left(\int_{0}^{t_{1,\varepsilon}} + \int_{t_{1,\varepsilon}}^{t_{2,\varepsilon}} + \int_{t_{2,\varepsilon}}^{\infty} \right)\left|f(x-\gamma(t))\right|\,\frac{\textrm{d}t}{t^{1-\alpha}}=:I^{(1)}f(x)+I^{(2)}f(x)+I^{(3)}f(x).
\end{align*}
Without loss of generality, we can assume that $\gamma$ is increasing on $\mathbb{R}^+$. For $I^{(2)}f$, by a change of variables and Young's inequality, it is easy to see that there exists a positive constant $C_{\varepsilon}$ such that $$\left\|I^{(2)}f\right\|_{L^{q}(\mathbb{R})} \leq C_{\varepsilon} \|f\|_{L^{p}(\mathbb{R})}$$ holds for all $1\leq p\leq q\leq \infty$. Let $M_{\varepsilon}\in [\omega_{0,2}+\varepsilon, \omega_{\infty,1}-\varepsilon]$. For $I^{(1)}f$, by a changes of variables and Lemma \ref{lemma 2.1}, noticing the trivial fact that $(\gamma^{-1})'(t)\cdot\gamma'(\gamma^{-1}(t))=1$ and \eqref{eq:2.1}, we get the following estimate
\begin{align*}
I^{(1)}f(x)\leq \frac{1}{C^{(1)}_1} \int_{0}^{\gamma(t_{1,\varepsilon})}|f(x-\omega)|\cdot(\gamma^{-1}(\omega))^{\alpha}\,\frac{\textrm{d}\omega}{\omega} \leq \frac{1}{C^{(1)}_1} \int_{0}^{\gamma(t_{1,\varepsilon})}|f(x-\omega)|\,\frac{\textrm{d}\omega}{\omega^{1-{\frac{\alpha}{M_{\varepsilon}}}}}.
\end{align*}
Moreover, by the Hardy-Littlewood-Sobolev inequality in Stein \cite{S}, we can obtain that $$\left\|I^{(1)}f\right\|_{L^{q}(\mathbb{R})} \ls \|f\|_{L^{p}(\mathbb{R})}$$ for all $1<p<q<\infty$ satisfying $\frac{1}{p}-\frac{1}{q}=\frac{\alpha}{M_{\varepsilon}}$. Similarly, we can obtain the same estimate for $I^{(3)}f$. Since $\varepsilon$ was arbitrary, we deduce that \eqref{eq:2.5} holds for all $1<p<q<\infty$ satisfying $\frac{1}{p}-\frac{1}{q}=\frac{\alpha}{M}$.

On the other hand, suppose that the estimate \eqref{eq:2.5} holds, then ${\frac{1}{p}}-{\frac{1}{q}}={\frac{\alpha}{N}}$ with $N\in [{\omega_{0,1}},{\omega_{\infty,2}}]$. Actually, let the  dilation $f_{\delta}(x):=f(\delta x)$ with $\delta>0$, we see that $\|f_{\delta}\|_{L^{p}(\mathbb{R})}={\delta}^{-{\frac{1}{p}}}\|f\|_{L^{p}(\mathbb{R})}$. Moreover, by a change of variables and Lemma \ref{lemma 2.1}, it is clear that
\begin{align*}
I_{\alpha,\gamma}(f_{\delta})(x)&\gs \left(\int_{0}^{\delta t_{1,\varepsilon}}+ \int_{\delta t_{1,\varepsilon}}^{\delta t_{2,\varepsilon}}+ \int_{\delta t_{2,\varepsilon}}^{\infty}\right) |f(\delta x-\omega)|\left(\gamma^{-1}\left(\frac{\omega}{\delta}\right)\right)^{\alpha}\,\frac{\textrm{d}\omega}{\omega}  \\
&\gs \int_{0}^{\delta t_{1,\varepsilon}}|f(\delta x-\omega)|\left({{\frac{\omega}{\delta}}}\right)^{\frac{\alpha}{\omega_{0,1}-\varepsilon}}\,\frac{\textrm{d}\omega}{\omega}
  +\int_{\delta t_{2,\varepsilon}}^{+\infty}|f(\delta x-\omega)|\left({{\frac{\omega}{\delta}}}\right)^{\frac{\alpha}{\omega_{\infty,2}+\varepsilon}}\,\frac{\textrm{d}\omega}{\omega}.
\end{align*}
We then have
$$I_{\alpha,\gamma}(f_{\delta})(x)\gs\delta^{-\frac{\alpha}{\omega_{0,1}-\varepsilon}}\int_{0}^{\delta t_{1,\varepsilon}}|f(\delta x-\omega)|\,\frac{\textrm{d}\omega}{\omega^{1-{\frac{\alpha}{\omega_{0,1}-\varepsilon}}}},$$
which leads to
$$\left\|\int_{0}^{\delta t_{1,\varepsilon}}|f(\delta x-\omega)|\,\frac{\textrm{d}\omega}{\omega^{1-{\frac{\alpha}{\omega_{0,1}-\varepsilon}}}}\right\|_{L^{q}(\mathbb{R})}
  \ls\delta^{\frac{\alpha}{\omega_{0,1}-\varepsilon}}\|I_{\alpha,\gamma}(f_{\delta})\|_{L^{q}(\mathbb{R})}
  \ls\delta^{\frac{\alpha}{\omega_{0,1}-\varepsilon}}{\delta}^{-{\frac{1}{p}}}\|f\|_{L^{p}(\mathbb{R})}.$$
Furthermore, we claim that the following holds:
$$\left\|\int_{0}^{\delta t_{1,\varepsilon}}|f(x-\omega)|\,\frac{\textrm{d}\omega}{\omega^{1-{\frac{\alpha}{\omega_{0,1}-\varepsilon}}}}\right\|_{L^{q}(\mathbb{R})}
 \ls\delta^{\frac{\alpha}{\omega_{0,1}-\varepsilon}-\frac{1}{p}+ \frac{1}{q}}\|f\|_{L^{p}(\mathbb{R})}.$$
Let $\delta\rightarrow \infty$, then one has ${\frac{\alpha}{\omega_{0,1}-\varepsilon}}-{\frac{1}{p}}+{\frac{1}{q}}\geq0$. Otherwise, the right side is $0$, this yields a contradiction. Therefore, ${\frac{1}{p}}-{\frac{1}{q}}\leq{\frac{\alpha}{\omega_{0,1}-\varepsilon}}$. Since $\varepsilon$ was arbitrary, we conclude that $\frac{1}{p}-\frac{1}{q}\leq\frac{\alpha}{\omega_{0,1}}$.
Similarly, we also have
$$I_{\alpha,\gamma}(f_{\delta})(x)\geq\delta^{-\frac{\alpha}{\omega_{\infty,2}+\varepsilon}}\int_{\delta t_{2,\varepsilon}}^{+\infty}|f(\delta x-\omega)|\,\frac{\textrm{d}\omega}{\omega^{1-{\frac{\alpha}{\omega_{\infty,2}+\varepsilon}}}},$$
the above implies that $p$ and $q$ satisfying $\frac{1}{p}-\frac{1}{q}\geq\frac{\alpha}{\omega_{\infty,2}}$. Therefore, $p$ and $q$ satisfy
$\frac{1}{p}-\frac{1}{q}=\frac{\alpha}{N}$ with $N\in [\omega_{0,1}, \omega_{\infty,2}]$, this concludes the proof of Lemma \ref{lemma 2.2}.
\end{proof}

From the proof of Lemma \ref{lemma 2.2}, we have the following remark to show the equivalent condition of Lemma \ref{lemma 2.2} for a class of new curves.

\begin{remark}\label{remark 2.3}
Assume that $\gamma\in C^{1}(\mathbb{R}^+)$, $\lim\limits_{t\rightarrow 0^+}\gamma(t)=0$, and there exist positive constants $0<A\leq B<\infty$ and $0<a\leq b<\infty$ such that $\gamma(t)\approx t^{a}$, $t\in(0,A)$ and $\gamma(t)\approx t^{b}$, $t\in(B,\infty)$, which\footnote{We note that $\gamma$ can be taken as a polynomial with no constant term. In fact, by a translation, a general polynomial is also suitable.} implies that  $\omega_{0,1}=\omega_{0,2}=a$ and $\omega_{\infty,1}=\omega_{\infty,2}=b$. Then, there exists a positive constant $C$ such that for all $f\in L^{p}(\mathbb{R})$, the estimate
\begin{align*}
\left\|I_{\alpha,\gamma}f\right\|_{L^{q}(\mathbb{R})}\leq C\|f\|_{L^{p}(\mathbb{R})}
\end{align*}
holds, if and only if $p$ and $q$ satisfy
$${\frac{1}{p}}-{\frac{1}{q}}\in \left[{\frac{\alpha}{b}},{\frac{\alpha}{a}}\right].$$
In addition, if $a>b$, then $I_{\alpha,\gamma}$ has no the corresponding $L^{p}(\mathbb{R})\rightarrow L^{q}(\mathbb{R})$ boundedness.
\end{remark}

\subsection {Some restricted weak type estimates for $H_{\alpha,\gamma}$}

We now establish some restricted weak type estimates for $H_{\alpha,\gamma}$ in the following Proposition \ref{proposition 2.4}. The definition of the restricted weak type $(p,q,r)$ estimate can be found in Grafakos \cite{Graf}. A operator is of restricted weak type $(p,q,r)$, by density, which in fact implies the $L^p(\mathbb{R})\times L^q(\mathbb{R}) \rightarrow L^{r,\infty}(\mathbb{R})$ boundedness of this operator.

\begin{proposition}\label{proposition 2.4}
Let $\gamma $ be the same as in Lemma \ref{lemma 2.1} with $\alpha\leq\omega_{0,2}\leq  \omega_{\infty,1}$. Then $H_{\alpha,\gamma}(f,g)$  is of restricted weak types $({\frac{1}{\alpha}},\infty,\infty)$, $(1,\infty,{\frac{1}{1-\alpha}})$, $(\infty,{\frac{M}{\alpha}},\infty)$ and $(\infty,1,{\frac{M}{M-\alpha}})$ with $M\in (\omega_{0,2}, \omega_{\infty,1})$.
\end{proposition}

\begin{proof} We split the following four parts to finish our proofs:
\begin{enumerate}
  \item[$\star$] Consider the restricted weak type $({\frac{1}{\alpha}},\infty,\infty)$. Take $f=\chi_{A}$ and $g=\chi_{B}$, where $A, B\subset\mathbb{R}$ are measurable sets of finite measure, by the Exercises 1.4.1 in Grafakos \cite{Gra}, it implies
  \begin{align}\label{eq:2.6}
  \left\|H_{\alpha,\gamma}(\chi_{A},\chi_{B})\right\|_{L^{\infty}(\mathbb{R})}
  \lesssim \sup\limits_{x\in{\mathbb{R}}} \int_{x-A} |t|^{\alpha-1}\,\textrm{d}t \lesssim \int_{ |t|\lesssim |A| } |t|^{\alpha-1}\,\textrm{d}t \lesssim|A|^{\alpha}.
  \end{align}
  We then have that $H_{\alpha,\gamma}(f,g)$ is restricted weak type $({\frac{1}{\alpha}},\infty,\infty)$.

  \item[$\star$] Consider the restricted weak type $(1,\infty,{\frac{1}{1-\alpha}})$. From Lemma \ref{lemma 2.2} and Remark \ref{remark 2.3}, let us set $\gamma(t):=t$ in \eqref{eq:2.4}, we may obtain that
  \begin{align*}
\left\|I_{\alpha,t}f\right\|_{L^{q}(\mathbb{R})} \lesssim \|f\|_{L^{p}(\mathbb{R})}
  \end{align*}
holds, if and only if $p$ and $q$ satisfy
$${\frac{1}{p}}-{\frac{1}{q}}=\alpha,$$
where
\begin{align*}
I_{\alpha,t}f(x):=\int_{0}^{\infty}\left|f(x-t)\right|\,\frac{\textrm{d}t}{t^{1-\alpha}}.
 \end{align*}
Therefore,
\begin{align*}
  \left\|H_{\alpha,\gamma}(\chi_{A},\chi_{B})\right\|_{L^{\frac{1}{1-\alpha}}(\mathbb{R})}
  \lesssim \left\|I_{\alpha,t}\chi_{A}\right\|_{L^{\frac{1}{1-\alpha}}(\mathbb{R})}\lesssim  \left\|\chi_{A}\right\|_{L^{1}(\mathbb{R})} \lesssim|A|,
  \end{align*}
 which implies that $H_{\alpha,\gamma}(f,g)$ is restricted weak type $(1,\infty,{\frac{1}{1-\alpha}})$.

  \item[$\star$] Consider the restricted weak type $(\infty,{\frac{M}{\alpha}},\infty)$. It is easy to see that
  \begin{align}\label{eq:2.8}
  \left\|H_{\alpha,\gamma}(\chi_{A},\chi_{B})\right\|_{L^{\infty}(\mathbb{R})}
  \lesssim  \left\|I_{\alpha,\gamma}\chi_{B}\right\|_{L^{\infty}(\mathbb{R})} .
  \end{align}
  For any $\varepsilon>0$, as in Lemma \ref{lemma 2.2}, we can conclude that
  \begin{align}\label{eq:2.9}
  I_{\alpha,\gamma}\chi_{B}(x) \lesssim I_{\frac{\alpha}{M_{\varepsilon}},t}\chi_{B}(x) +I^{(2)}\chi_{B}(x),
  \end{align}
  where the definition of $I^{(2)}$ can be found in Lemma \ref{lemma 2.2}. Moreover, by \eqref{eq:2.6}, we obtain that $\|I_{\frac{\alpha}{M_{\varepsilon}},t}\chi_{B}\|_{L^{\infty}(\mathbb{R})}$ can be bounded from above by $|B|^{\frac{\alpha}{M_{\varepsilon}}}$, and the same estimate also holds for $I^{(2)}\chi_{B}$ since $\omega_{0,2}\geq \alpha$. Therefore, from \eqref{eq:2.8} and \eqref{eq:2.9}, it follows that
  \begin{align*}
  \left\|H_{\alpha,\gamma}(\chi_{A},\chi_{B})\right\|_{L^{\infty}(\mathbb{R})}
  \lesssim  |B|^{\frac{\alpha}{M_{\varepsilon}}}.
  \end{align*}
  Since $\varepsilon$ was arbitrary, we deduce that $H_{\alpha,\gamma}(f,g)$ is restricted weak type $(\infty,{\frac{M}{\alpha}},\infty)$.

  \item[$\star$] Consider the restricted weak type $(\infty,1,{\frac{M}{M-\alpha}})$. By Lemma \ref{lemma 2.2}, we can see that
 \begin{align*}
\left \|H_{\alpha,\gamma}(\chi_{A},\chi_{B})\right\|_{L^{\frac{M}{M-\alpha}}(\mathbb{R})}
\lesssim  \left \|I_{\alpha,\gamma}\chi_{B}\right\|_{L^{\frac{M}{M-\alpha}}(\mathbb{R})}\lesssim  \left \|\chi_{B}\right\|_{L^1(\mathbb{R})}\lesssim |B|,
\end{align*}
which further implies that $H_{\alpha,\gamma}(f,g)$ is restricted weak type $(\infty,1,{\frac{M}{M-\alpha}})$.
\end{enumerate}
Therefore, putting things together we finish the proof of Proposition \ref{proposition 2.4}.
\end{proof}

As in Remark \ref{remark 2.3}, we also have the following remark.

\begin{remark}\label{remark 2.5}
Let $\gamma $ be the same as in Remark \ref{remark 2.3}. Then $H_{\alpha,\gamma}(f,g)$  is of restricted weak types $({\frac{1}{\alpha}},\infty,\infty)$, $(1,\infty,{\frac{1}{1-\alpha}})$, $(\infty,{\frac{M}{\alpha}},\infty)$ and $(\infty,1,{\frac{M}{M-\alpha}})$ with $M\in [a, b]$.
\end{remark}

\section{A uniform estimate for $H_{\alpha,\gamma,j}$}

In this section, we devote to establish a uniform estimate for $H_{\alpha,\gamma,j}$, which comes from a decomposition of $H_{\alpha,\gamma}$. Indeed, without loss of generality, we can assume $f, g\geq 0$ and decompose
 \begin{align*}
H_{\alpha,\gamma}(f,g)(x)=\sum_{j \in \mathbb{Z}}\int_{2^{j-1}}^{2^j}f(x-t)g(x-\gamma(t))\,\frac{\textrm{d}t}{t^{1-\alpha}}\lesssim \sum_{j \in \mathbb{Z}}2^{(\alpha-1)j}H_{\alpha,\gamma,j}(f,g)(x),
\end{align*}
where
\begin{align}\label{eq:3.1}
H_{\alpha,\gamma,j}(f,g)(x):=\int_{2^{j-1}}^{2^j}f(x-t)g(x-\gamma(t))\,\textrm{d}t.
\end{align}

For estimating $H_{\alpha,\gamma,j}$, it is natural to use the theory of oscillatory integrals and stationary phase estimates. From $\lim\limits_{t\rightarrow 0^+}\gamma'(0)=0$ and \eqref{eq:1.1}, it implies that $|\gamma'(t)|$ is strictly increasing on $t\in\mathbb{R}^+$. As in Liu, Song and Yu \cite[Remark 1.4]{LSY}, by letting $G(t):=\ln |\gamma'(t)|$, we claim that
\begin{align}\label{eq:3.2}
 \frac{|\gamma'(2t)|}{|\gamma'(t)|}\geq e^{C_0 }
  \end{align}
holds for all $t\in \mathbb{R}^+$ with $C_0:=\frac{C^{(2)}_{1}}{2C^{(1)}_{2}}$, which further implies that $\lim\limits_{t\rightarrow \infty}|\gamma'(t)|=\infty$. Therefore, we may assume that $|\gamma'(1)|>1 $ by the dilation $f_{\delta}(x):=f(\delta x)$ and \eqref{eq:1.1}. Furthermore, the change of variables $u=x-t$ and $v=x-\gamma(t)$ yields that the Jocobian ${\frac{\partial(u,v)}{\partial(x,t)}}=1-\gamma'(t)$, and then there exists at most one stationary point $t_0$ satisfying $\gamma'(t_0)=1$ and $t_0\in(2^{j_0},2^{j_{0}+1}]$ for some negative integer $j_0$ small enough. We shall consider two cases: non-critical case and critical case\footnote{If there does not exist the stationary point $t_0$, this is identical to the non-critical case.} in the following Propositions \ref{proposition 3.1} and \ref{proposition 3.2}, respectively.

\begin{proposition}\label{proposition 3.1} (Non-critical case $j\neq j_0, j_0+1, j_0+2$) Assume that $\gamma\in C^2(\mathbb{R}^+)$ is monotonic on $\mathbb{R}^+$, $\lim\limits_{t\rightarrow 0^+}\gamma(t)=\lim\limits_{t\rightarrow 0^+}\gamma'(0)=0$, $\omega_{0,1}, \omega_{\infty,1}> 1$, and satisfy \eqref{eq:1.1}. Then, there exists a positive constant $C$, such that
$$\left\|H_{\alpha,\gamma,j}(f,g)\right\|_{L^{r}(\mathbb{R})}\leq C 2^j\|f\|_{L^{p}(\mathbb{R})}\|g\|_{L^{q}(\mathbb{R})}$$
 for all $f\in L^{p}(\mathbb{R})$ and $g\in L^{q}(\mathbb{R})$, where $p,q,r$ satisfy $\frac{1}{p}+\frac{1}{q}=\frac{1}{r}$ and $p, q \in [1, \infty]$.
\end{proposition}

\begin{proof} We first consider the $L^1(\mathbb{R})\times L^1(\mathbb{R}) \rightarrow L^{\frac{1}{2}}(\mathbb{R})$ estimate for $H_{\alpha,\gamma,j}$. If $j<j_0$, then
$$\left|{\frac{\partial(u,v)}{\partial(x,t)}}\right|=|1-\gamma'(t)|\geq 1-|\gamma'(t)|>1-|\gamma'(2^{j_0})|\gs 1,$$
and furthermore, the change of variables $u=x-t$ and $v=x-\gamma(t)$ yields that
\begin{align}\label{eq:3.3}
\left\|H_{\alpha,\gamma,j}(f,g)\right\|_{L^{1}(\mathbb{R})}\lesssim  \int_\mathbb{R}\int_\mathbb{R} f(u)g(v){\frac{1}{|1-\gamma'(t)|}}\,\textrm{d}u\textrm{d}v\lesssim \|f\|_{L^{1}(\mathbb{R})}\|g\|_{L^{1}(\mathbb{R})}.
 \end{align}
On the other hand, for any measurable set $E$ with finite measure, apply H\"{o}lder's inequality and \eqref{eq:3.3} to deduce that
\begin{align}\label{eq:3.4}
\left\|H_{\alpha,\gamma,j}(f,g)\chi_{E}\right\|_{L^{\frac{1}{2}}(\mathbb{R})}\leq |E|\left\|H_{\alpha,\gamma,j}(f,g)\right\|_{L^{1}(\mathbb{R})}\lesssim |E|\|f\|_{L^{1}(\mathbb{R})}\|g\|_{L^{1}(\mathbb{R})}.
\end{align}

Let the interval $I_k:=((k-1)2^j,k2^j]$ with $k\in \mathbb{Z}$. It is easy to check that $H_{\alpha,\gamma,j}(f \chi_{I_{k}},g\chi_{I_{k+d}})(x)$ is supported in a finitely many terms of these intervals and $H_{\alpha,\gamma,j}(f \chi_{I_{k}},g\chi_{I_{k+d}})(x)=0$ when $|d| >3$ for every fixed $k\in \mathbb{Z}$. Therefore, by \eqref{eq:3.4}, we conclude that
\begin{align*}
\left\|H_{\alpha,\gamma,j}(f \chi_{I_{k}},g\chi_{I_{k+d}})\right\|_{L^{\frac{1}{2}}(\mathbb{R})}\lesssim 2^j\left\|f \chi_{I_{k}}\right\|_{L^{1}(\mathbb{R})}\left\|g\chi_{I_{k+d}}\right\|_{L^{1}(\mathbb{R})},
\end{align*}
which, together with the triangle inequality and the Cauchy-Schwarz inequality, shows that
\begin{align}\label{eq:3.4a}
\left\|H_{\alpha,\gamma,j}(f,g)\right\|_{L^{\frac{1}{2}}(\mathbb{R})} &=\left\|\sum\limits_{d\in\mathbb{Z},~|d|\leq 3}\sum\limits_{k\in\mathbb{Z}}H_{\alpha,\gamma,j}(f \chi_{I_{k}},g\chi_{I_{k+d}})\right\|_{L^{\frac{1}{2}}(\mathbb{R})} \\
&\leq \sum\limits_{d\in\mathbb{Z},~|d|\leq 3} \left[\sum\limits_{k\in\mathbb{Z}} \left\|H_{\alpha,\gamma,j}(f \chi_{I_{k}},g\chi_{I_{k+d}})\right\|^{\frac{1}{2}}_{L^{\frac{1}{2}}(\mathbb{R})}    \right]^2\nonumber\\
&\lesssim 2^j \sum\limits_{d\in\mathbb{Z},~|d|\leq 3} \left[\sum\limits_{k\in\mathbb{Z}} \left(\left\|f \chi_{I_{k}}\right\|_{L^{1}(\mathbb{R})}\left\|g\chi_{I_{k+d}}\right\|_{L^{1}(\mathbb{R})} \right)^{\frac{1}{2}}  \right]^2\lesssim 2^j\|f\|_{L^{1}(\mathbb{R})}\|g\|_{L^{1}(\mathbb{R})}.\nonumber
\end{align}
This is the desired $L^1(\mathbb{R})\times L^1(\mathbb{R}) \rightarrow L^{\frac{1}{2}}(\mathbb{R})$ estimate for $H_{\alpha,\gamma,j}$ in the case of $j<j_0$.

We now turn to the case of $j>j_{0}+2$. An easy calculation gives
$$\left|{\frac{\partial(u,v)}{\partial(x,t)}}\right|\geq|\gamma'(t)|-1\geq|\gamma'(2^{j-1})|-1\geq|\gamma'(2^{j_0+2})|-|\gamma'(2^{j_0+1})|\gs 1.$$
Taking this into account one can see that
\begin{align*}
\left\|H_{\alpha,\gamma,j}(f,g)\right\|_{L^{1}(\mathbb{R})}\lesssim \frac{1}{|\gamma'(2^{j-1})|-1}\|f\|_{L^{1}(\mathbb{R})}\|g\|_{L^{1}(\mathbb{R})}.
 \end{align*}
As in \eqref{eq:3.4}, noticing that $|\gamma'(2^{j-1})|-1 \gs |\gamma'(2^{j-1})|$ for $j>j_{0}+2$, thus we have an estimate
\begin{align}\label{eq:3.5}
\left\|H_{\alpha,\gamma,j}(f,g)\chi_{E}\right\|_{L^{\frac{1}{2}}(\mathbb{R})}\lesssim \frac{|E|}{|\gamma'(2^{j-1})|}\|f\|_{L^{1}(\mathbb{R})}\|g\|_{L^{1}(\mathbb{R})}.
 \end{align}

From $\lim\limits_{t\rightarrow 0^+}\gamma(0)=0$ and \eqref{eq:1.1}, it shows that $|\gamma(t)|$ is strictly increasing on $t\in\mathbb{R}^+$. We also from \eqref{eq:2.2} have $\lim\limits_{t\rightarrow \infty}|\gamma(t)|=\infty$. Then, we can split $\mathbb{R}$ as a  union of intervals $J_k:=((k-1)|\gamma(2^j)|,k|\gamma(2^j)|]$ with $k\in \mathbb{Z}$. Furthermore, one sees that $H_{\alpha,\gamma,j}(f \chi_{J_{k}},g\chi_{J_{k+d}})(x)$ is supported in a finitely many terms of these intervals, and there exists a universal constant $C$ such that $H_{\alpha,\gamma,j}(f \chi_{J_{k}},g\chi_{J_{k+d}})(x)=0$ when $|d| >C$ for every fixed $k\in \mathbb{Z}$. Indeed, we may take $C:=\max \{-2+\frac{C^{(1)}_2}{2},1+2^{{j_0}+2}\}$. Therefore, by \eqref{eq:3.5}, \eqref{eq:2.2} and \eqref{eq:1.1}, we obtain
\begin{align*}
\left\|H_{\alpha,\gamma,j}(f \chi_{J_{k}},g\chi_{J_{k+d}})\right\|_{L^{\frac{1}{2}}(\mathbb{R})}\lesssim \frac{|\gamma(2^j)|}{|\gamma'(2^{j-1})|} \left\|f \chi_{J_{k}}\right\|_{L^{1}(\mathbb{R})}\left\|g\chi_{J_{k+d}}\right\|_{L^{1}(\mathbb{R})}\lesssim 2^j\left\|f \chi_{J_{k}}\right\|_{L^{1}(\mathbb{R})}\left\|g\chi_{J_{k+d}}\right\|_{L^{1}(\mathbb{R})}.
\end{align*}
Furthermore, as in \eqref{eq:3.4a}, we also assert that
\begin{align}\label{eq:3.50}
\left\|H_{\alpha,\gamma,j}(f,g)\right\|_{L^{\frac{1}{2}}(\mathbb{R})}\lesssim2^j\|f\|_{L^{1}(\mathbb{R})}\|g\|_{L^{1}(\mathbb{R})}
\end{align}
 holds for all $j>j_{0}+2$.

Putting things together, for the non-critical case $j\neq j_0, j_0+1, j_0+2$, we have established the $L^1(\mathbb{R})\times L^1(\mathbb{R}) \rightarrow L^{\frac{1}{2}}(\mathbb{R})$ estimate for $H_{\alpha,\gamma,j}$. The remaining $L^p(\mathbb{R})\times L^q(\mathbb{R}) \rightarrow L^r(\mathbb{R})$ estimates follow by interpolation with the following three trivial estimates: $L^\infty(\mathbb{R})\times L^\infty(\mathbb{R}) \rightarrow L^\infty(\mathbb{R})$, $L^\infty(\mathbb{R})\times L^1(\mathbb{R}) \rightarrow L^1(\mathbb{R})$ and $L^1(\mathbb{R})\times L^\infty(\mathbb{R}) \rightarrow L^1(\mathbb{R})$. Therefore, we finish the proof of Proposition \ref{proposition 3.1}.
\end{proof}

\begin{proposition}\label{proposition 3.2} (Critical cases $j= j_0, j_0+1, j_0+2$) Let $\gamma $ be the same as in Proposition \ref{proposition 3.1}. Then, there exists a positive constant $C$, such that
$$\left\|H_{\alpha,\gamma,j}(f,g)\right\|_{L^{r}(\mathbb{R})}\leq C \|f\|_{L^{p}(\mathbb{R})}\|g\|_{L^{q}(\mathbb{R})}$$
for all $f\in L^{p}(\mathbb{R})$ and $g\in L^{q}(\mathbb{R})$, where $p,q,r$ satisfy $\frac{1}{p}+\frac{1}{q}=\frac{1}{r}$, $p, q \in [1, \infty]$ and $(p,q)\neq (1,1)$.
\end{proposition}

\begin{proof}
We only consider the case $j= j_0$, the proofs of the cases $j=j_0+1, j_0+2$ are similar to the case $j= j_0$ and we omit the details. Since $j_0$ is a negative integer, we can bound $H_{\alpha,\gamma,j_0}(f,g)(x)$ by
$$\int_{0}^{1}f(x-t)g(x-\gamma(t))\,\textrm{d}t.$$
Moreover, we can restrict $x$ into the interval $[0,1]$. Indeed, let $f_N(x):=f(x+N)$ and $g_N(x):=g(x+N)$ if we have the restriction of $x$, then $\|H_{\alpha,\gamma,j_0}(f,g)\|^r_{L^{r}(\mathbb{R})}$ can be dominated by a constant times
\begin{align*}
\sum\limits_{N\in\mathbb{Z}}\int_{0}^{1}\left(\int_{0}^{1}f_N(x-t)g_N(x-\gamma(t))\,\textrm{d}t\right)^r\,\textrm{d}x
\lesssim \sum_{N\in\mathbb{Z}}\left\|f_{N}\right\|_{L^{p}(\mathbb{R})}^{r}\left\|g_{N}\right\|_{L^{q}(\mathbb{R})}^{r}
\lesssim \|f\|_{L^{p}(\mathbb{R})}^{r}\|g\|_{L^{q}(\mathbb{R})}^{r},
\end{align*}
where the last inequality follows from H\"{o}lder's inequality.

Therefore, it suffices to show that
\begin{align}\label{eq:3.7}
\int_{0}^{1}\left(\int_{0}^{1}f(x-t)g(x-\gamma(t))\,\textrm{d}t\right)^r\,\textrm{d}x\lesssim \|f\|_{L^{p}(\mathbb{R})}^{r}\|g\|_{L^{q}(\mathbb{R})}^{r}
\end{align}
for all $p,q,r$ satisfying $\frac{1}{p}+\frac{1}{q}=\frac{1}{r}$, $p, q \in [1, \infty]$ and $(p,q)\neq (1,1)$.

When $r\geq1$, it is easy to see that \eqref{eq:3.7} follows from Minkowski's integral inequality and H\"{o}lder's inequality. When $\frac{1}{2}<r<1$, we need some more complicated calculations. Indeed, let $\psi$ be a smooth function supported on $\{t\in \mathbb{R}:\ \frac{1}{2}\leq |t|\leq 2\}$ with the property that $0\leq \psi(t)\leq 1$ and $\Sigma_{k\in \mathbb{Z}} \psi_k(t)=1$, where $\psi_k(t):=\psi (2^{-k}t)$. In order to obtain \eqref{eq:3.7} with $\frac{1}{2}<r<1$, it suffices to establish the following estimate: there exists a positive constant $\varepsilon$ such that
\begin{align}\label{eq:3.8}
\int_{0}^{1}\left[\int_{0}^{1}f(x-t)g(x-\gamma(t))\psi\left(2^{k}(t-t_{0})\right)\,\textrm{d}t\right]^r\,\textrm{d}x\lesssim 2^{-\varepsilon k}\|f\|_{L^{p}(\mathbb{R})}^{r}\|g\|_{L^{q}(\mathbb{R})}^{r}
\end{align}
for all $k>0$, where $p,q,r$ satisfy $\frac{1}{p}+\frac{1}{q}=\frac{1}{r}$, $p, q \in [1, \infty]$ and $\frac{1}{2}<r<1$.

We establish \eqref{eq:3.8} by considering the following two cases:
\begin{enumerate}
  \item[$\bullet$] The case of $2^{-k+2}\leq t_0$. Notice that ${\frac{1}{2}}\cdot2^{-k}\leq|t-t_0|\leq2\cdot2^{-k}$, it implies $|t-t_0|\leq{\frac{1}{2}}t_0$ and we can then rewrite the left hand side of \eqref{eq:3.8} as
\begin{align}\label{eq:3.9}
\int_{0}^{1}\left[\int_{-\frac{t_0}{2}}^{\frac{t_0}{2}}f_{t_0}(x-t)g_{t_0}\left(x-\Upsilon_{t_0}(t)\right)\psi\left(2^{k}t\right)\,\textrm{d}t\right]^r\,\textrm{d}x,
\end{align}
where $$f_{t_0}(x):=f(x-t_0), ~~g_{t_0}(x):=g(x-\gamma(t_0))~~ \textrm{and} ~~\Upsilon_{t_0}(t):=\gamma(t+t_{0})-\gamma(t_{0}).$$ On the other hand, we obviously have $|t|\leq 2^{-k+1}$. Moreover, by Lagrange's mean value theorem, $|\Upsilon_{t_0}(t)|\leq |\gamma'(\zeta)||t|\leq 2|\gamma'(2)|2^{-k}$, where $\zeta\in({\frac{t_0}{2}},{\frac{3}{2}}t_0)\subset(0,2)$. Then, if we set
\begin{align*}
\Omega_{N}:=\left(2N|\gamma'(2)|2^{-k}, 2(N+1)|\gamma'(2)|2^{-k}\right]
\end{align*}
with $N=0,1,\ldots,\lceil\frac{2^k}{2|\gamma'(2)|}\rceil$, which trivially leads to $x-t, x-\Upsilon_{t_0}(t)\in \Omega_{N-1}\cup \Omega_{N} \cup \Omega_{N+1}$ for all $x\in \Omega_{N}$.

\quad Therefore, to obtain \eqref{eq:3.8}, it is enough to show that, there exists a positive constant $\varepsilon$ such that
\begin{align}\label{eq:3.10}
\left\|\Lambda_N(f_{t_0, N}, g_{t_0, N})\right\|_{L^{r}(\mathbb{R})}^{r}\lesssim 2^{-\varepsilon k}\|f_{t_0, N}\|_{L^{p}(\mathbb{R})}^{r}\|g_{t_0, N}\|_{L^{q}(\mathbb{R})}^{r}
\end{align}
for all $k>0$, where $f_{t_0, N}:=f_{t_0}\chi_{\Omega_{N}}$, $g_{t_0, N}:=g_{t_0}\chi_{\Omega_{N}}$ and
\begin{align*}
\Lambda_N(f_{t_0, N}, g_{t_0, N})(x):=\chi_{\Omega_{N}}(x)\int_{-\frac{t_0}{2}}^{\frac{t_0}{2}}f_{t_0, N}(x-t)g_{t_0, N}\left(x-\Upsilon_{t_0}(t)\right)\psi\left(2^{k}t\right)\,\textrm{d}t.
\end{align*}
Indeed, notice that $\frac{1}{2}<r<1$, from \eqref{eq:3.9}, by the triangle inequality, we can bound the left hand side of \eqref{eq:3.8} by $$\sum_{N\in [0, \lceil\frac{2^k}{2|\gamma'(2)|}\rceil]} \left\|\Lambda_N(f_{t_0, N}, g_{t_0, N})\right\|_{L^{r}(\mathbb{R})}^{r}$$ essentially. This, in combination with \eqref{eq:3.10} and H\"{o}lder's inequality, leads to the desired estimate \eqref{eq:3.8}.

\quad We now turn to the proof of \eqref{eq:3.10} by establishing the following estimates:
\begin{enumerate}
  \item[$\circ$] The $L^1(\mathbb{R})\times L^\infty(\mathbb{R}) \rightarrow L^1(\mathbb{R})$ boundedness of $\Lambda_N(f_{t_0, N}, g_{t_0, N})$. A calculation gives
\begin{align*}
\left\|\Lambda_N(f_{t_0, N}, g_{t_0, N})\right\|_{L^{1}(\mathbb{R})}&\leq \|g_{t_0, N}\|_{L^{\infty}(\mathbb{R})}\int_{0}^{1}\chi_{\Omega_{N}}(x)\int_{-\frac{t_0}{2}}^{\frac{t_0}{2}}f_{t_0, N}(x-t)\psi\left(2^{k}t\right)\,\textrm{d}t\,\textrm{d}x\\ & \lesssim 2^{- k}\|f_{t_0, N}\|_{L^{1}(\mathbb{R})}\|g_{t_0, N}\|_{L^{\infty}(\mathbb{R})}.
\end{align*}

  \item[$\circ$] The $L^{\infty}(\mathbb{R})\times L^1(\mathbb{R}) \rightarrow L^1(\mathbb{R})$ boundedness of $\Lambda_N(f_{t_0, N}, g_{t_0, N})$. One easily sees that
  \begin{align*}
\left\|\Lambda_N(f_{t_0, N}, g_{t_0, N})\right\|_{L^{1}(\mathbb{R})}&\leq \|f_{t_0, N}\|_{L^{\infty}(\mathbb{R})}\int_{-\frac{t_0}{2}}^{\frac{t_0}{2}}\int_{0}^{1}\chi_{\Omega_{N}}(x)g_{t_0, N}\left(x-\Upsilon_{t_0}(t)\right)\,\textrm{d}x ~\psi\left(2^{k}t\right) \,\textrm{d}t\\ & \lesssim 2^{- k}\|f_{t_0, N}\|_{L^{\infty}(\mathbb{R})}\|g_{t_0, N}\|_{L^{1}(\mathbb{R})}.
\end{align*}

  \item[$\circ$] The $L^1(\mathbb{R})\times L^1(\mathbb{R}) \rightarrow L^{\frac{1}{2}}(\mathbb{R})$ boundedness of $\Lambda_N(f_{t_0, N}, g_{t_0, N})$. As in Proposition \ref{proposition 3.1}, we need a $L^1(\mathbb{R})\times L^1(\mathbb{R}) \rightarrow L^{1}(\mathbb{R})$ estimate for $\Lambda_N(f_{t_0, N}, g_{t_0, N})$. By Lagrange's mean value theorem, it follows from the change of variables $u=x-t$ and $v=x-\Upsilon_{t_0}(t)$ that $$\left|{\frac{\partial(u,v)}{\partial(x,t)}}\right|=|1-\Upsilon_{t_0}'(t)|=|\gamma'(t+t_{0})-\gamma'(t_{0})|=|\gamma''(\vartheta)||t|\gs 2^{-k},$$
where $\vartheta \in [{\frac{1}{2}}t_0, {\frac{3}{2}}t_0]$. Then,
  \begin{align*}
\left\|\Lambda_N(f_{t_0, N}, g_{t_0, N})\right\|_{L^{1}(\mathbb{R})}&\lesssim  \int_\mathbb{R}\int_\mathbb{R} f_{t_0, N}(u)g_{t_0, N}(v)\left|{\frac{\partial(x,t)}{\partial(u,v)}}\right|\,\textrm{d}u\textrm{d}v\\ &\lesssim 2^{k} \|f_{t_0, N}\|_{L^{1}(\mathbb{R})}\|g_{t_0, N}\|_{L^{1}(\mathbb{R})}.
\end{align*}
Moreover, one may use the Cauchy-Schwartz inequality to deduce that
  \begin{align*}
\left\|\Lambda_N(f_{t_0, N}, g_{t_0, N})\right\|_{L^{\frac{1}{2}}(\mathbb{R})}\lesssim |\Omega_{N}|\left\|\Lambda_N(f_{t_0, N}, g_{t_0, N})\right\|_{L^{1}(\mathbb{R})} \lesssim \|f_{t_0, N}\|_{L^{1}(\mathbb{R})}\|g_{t_0, N}\|_{L^{1}(\mathbb{R})}.
\end{align*}
\end{enumerate}
Putting things together, it follows from an interpolation that \eqref{eq:3.10} holds for all $p,q,r$ satisfying $\frac{1}{p}+\frac{1}{q}=\frac{1}{r}$, $p, q \in [1, \infty]$ and $\frac{1}{2}<r<1$.

  \item[$\bullet$] The case of $2^{-k+2}> t_0$. Let
  $$\Pi_k(f,g)(x):=\int_{0}^{1}f(x-t)g(x-\gamma(t))\psi\left(2^{k}(t-t_{0})\right)\,\textrm{d}t,$$
  it follows from $2^{-k+2}> t_0$ that $k< \log_2 (\frac{4}{t_0})$, therefore to prove \eqref{eq:3.8} it will suffice to prove
   \begin{align*}
\left\|\Pi_k(f,g)\right\|_{L^{r}([0, 1])}\lesssim \|f\|_{L^{p}(\mathbb{R})}\|g\|_{L^{q}(\mathbb{R})}
\end{align*}
for all $p,q,r$ satisfying $\frac{1}{p}+\frac{1}{q}=\frac{1}{r}$, $p, q \in [1, \infty]$ and $\frac{1}{2}<r<1$. This can be proved by an interpolation with the following estimates:
\begin{enumerate}
  \item[$\circ$] The $L^1(\mathbb{R})\times L^\infty(\mathbb{R}) \rightarrow L^1([0, 1])$ boundedness of $\Pi_k(f,g)$. It is easy to get
\begin{align*}
\left\|\Pi_k(f,g)\right\|_{L^{1}([0, 1])}&\leq \|g\|_{L^{\infty}(\mathbb{R})}\int_{0}^{1}\int_{0}^{1} f(x-t)\psi\left(2^{k}(t-t_{0})\right)\,\textrm{d}t\,\textrm{d}x \lesssim \|f\|_{L^{1}(\mathbb{R})}\|g\|_{L^{\infty}(\mathbb{R})}.
\end{align*}

  \item[$\circ$] The $L^{\infty}(\mathbb{R})\times L^1(\mathbb{R}) \rightarrow L^1([0, 1])$ boundedness of $\Pi_k(f,g)$. One obtains
  \begin{align*}
\left\|\Pi_k(f,g)\right\|_{L^{1}([0, 1])}&\leq \|f\|_{L^{\infty}(\mathbb{R})}\int_{0}^{1}\int_{0}^{1}g(x-\gamma(t))\psi\left(2^{k}(t-t_{0})\right)\,\textrm{d}t\,\textrm{d}x \lesssim \|f\|_{L^{\infty}(\mathbb{R})}\|g\|_{L^{1}(\mathbb{R})}.
\end{align*}

  \item[$\circ$] The $L^1(\mathbb{R})\times L^1(\mathbb{R}) \rightarrow L^{\frac{1}{2}}([0, 1])$ boundedness of $\Pi_k(f,g)$. From ${\frac{1}{2}}\cdot2^{-k}\leq|t-t_0|\leq2\cdot2^{-k}$, $2^{-k+2}> t_0$ and $0<t<1$, which means that $0<t\leq \frac{7}{8}t_0$ or $ \frac{9}{8}t_0\leq t<1$. Furthermore, by the facts that $\gamma'(t_0)=1$ and $|\gamma'(t)|$ is strictly increasing on $t\in\mathbb{R}^+$, the change of variables $u=x-t$ and $v=x-\gamma(t)$ implies that
     \begin{align*}
     \left|{\frac{\partial(u,v)}{\partial(x,t)}}\right|&=|1-\gamma'(t)|=|\gamma'(t)-\gamma'(t_{0})|\\
     &\geq
\begin{cases}|\gamma'(t)|-|\gamma'(t_{0})|\geq \left|\gamma'\left(\frac{9}{8}t_0\right)\right|-\left|\gamma'(t_{0})\right|\gs 1, ~~~\textrm{if} ~\frac{9}{8}t_0\leq t<1;\\
|\gamma'(t_{0})|-|\gamma'(t)|\geq \left|\gamma'(t_{0})\right|-\left|\gamma'\left(\frac{7}{8}t_0\right)\right|\gs 1, ~~~\textrm{if} ~0<t\leq \frac{7}{8}t_0.
\end{cases}
\end{align*}
Hence, we have
  \begin{align*}
\left\|\Pi_k(f,g)\right\|_{L^{1}([0, 1])}&\lesssim  \int_\mathbb{R}\int_\mathbb{R} f(u)g(v)\left|{\frac{\partial(x,t)}{\partial(u,v)}}\right|\,\textrm{d}u\textrm{d}v\lesssim  \|f\|_{L^{1}(\mathbb{R})}\|g\|_{L^{1}(\mathbb{R})}.
\end{align*}
This, together with the Cauchy-Schwartz inequality, leads to
  \begin{align*}
\left\|\Pi_k(f,g)\right\|_{L^{\frac{1}{2}}([0, 1])}\lesssim\left\|\Pi_k(f,g)\right\|_{L^{1}([0, 1])}\lesssim \|f\|_{L^{1}(\mathbb{R})}\|g\|_{L^{1}(\mathbb{R})}.
\end{align*}
\end{enumerate}
\end{enumerate}

Therefore, we finish the proof of the Proposition \ref{proposition 3.2}.
\end{proof}

\section{Estimates for $H_{\alpha,\gamma}$ with $f\in L^1(\mathbb{R})$}

Recall that
\begin{align}\label{eq:4.1}
H_{\alpha,\gamma}(f,g)(x)\lesssim \sum_{j \in \mathbb{Z}}2^{(\alpha-1)j}H_{\alpha,\gamma,j}(f,g)(x).
\end{align}
Our goal in this section is to establish some $L^p(\mathbb{R})\times L^q(\mathbb{R}) \rightarrow L^{r,\infty}(\mathbb{R})$ estimates for $H_{\alpha,\gamma}$ with $p=1$ by considering the corresponding estimates for $H_{\alpha,\gamma,j}$.

\begin{proposition}\label{proposition 4.1}
Let $\gamma $ be the same as in Theorem \ref{maintheorem}. Then, there exists a positive constant $C$, such that
$$\left\|H_{\alpha,\gamma}(f,g)\right\|_{L^{r(\varepsilon),\infty}(\mathbb{R})}\leq C \|f\|_{L^{1}(\mathbb{R})}\left\|g\right\|_{L^{2r(\varepsilon)}(\mathbb{R})},$$
where $\frac{1}{r( \varepsilon ) }:=\frac{2(\omega_{\infty,2}-\alpha)}{2\omega_{\infty,2}-\omega_{\infty,1}}-\varepsilon $ with $\varepsilon>0$ small enough.
\end{proposition}

\begin{proof}
For $H_{\alpha,\gamma,j}$ defined in \eqref{eq:3.1}, let $E$ be a measurable set with finite measure, a calculation gives
\begin{align*}
 \left \|H_{\alpha,\gamma,j}(f,g)\chi_{E}\right\|_{L^{1}(\mathbb{R})}
  \leq\|g\|_{L^{\infty}(\mathbb{R})}\int_{E}\int_{2^{j-1}}^{2^j}f(x-t)\,\textrm{d}t\,\textrm{d}x
  \leq|E|\|f\|_{L^{1}(\mathbb{R})}\|g\|_{L^{\infty}(\mathbb{R})}
 \end{align*}
for all $j\in \mathbb{Z}$. This, combined with Propositions \ref{proposition 3.1} and \ref{proposition 3.2}, shows that
  \begin{align}\label{eq:4.2}
  \left\|H_{\alpha,\gamma,j}(f,g)\chi_{E}\right\|_{L^{1}(\mathbb{R})}\lesssim \min\left\{2^j,|E|\right\}\|f\|_{L^{1}(\mathbb{R})}\|g\|_{L^{\infty}(\mathbb{R})}
  \end{align}
for all $j\in \mathbb{Z}$. From \eqref{eq:3.4} and \eqref{eq:3.4a}, we can write
  \begin{align}\label{eq:4.3}
  \left\|H_{\alpha,\gamma,j}(f,g)\chi_{E}\right\|_{L^{\frac{1}{2}}(\mathbb{R})}\lesssim \min\left\{2^j,|E|\right\}\|f\|_{L^{1}(\mathbb{R})}\|g\|_{L^{1}(\mathbb{R})}
  \end{align}
  for all $j<j_0$, and by \eqref{eq:3.5} and \eqref{eq:3.50}, it follows that
   \begin{align}\label{eq:4.4}
  \left\|H_{\alpha,\gamma,j}(f,g)\chi_{E}\right\|_{L^{\frac{1}{2}}(\mathbb{R})}\lesssim \min\left\{2^j,\frac{|E|}{|\gamma'(2^{j-1})|}\right\}\|f\|_{L^{1}(\mathbb{R})}\|g\|_{L^{1}(\mathbb{R})}
  \end{align}
 for all $j>j_0+2$.

Now, for any fixed $f\in L^1(\mathbb{R})$, the operator $H_{\alpha,\gamma,j}(f,g)\chi_{E}$ may be viewed as a linear operator about $g$. Interpolation \eqref{eq:4.2} with \eqref{eq:4.3} implies that
 \begin{align}\label{eq:4.5}
  \left\|H_{\alpha,\gamma,j}(f,g)\chi_{E}\right\|_{L^{\frac{1}{1+\theta_1}}(\mathbb{R})}\lesssim \min\left\{2^j,|E|\right\}\|f\|_{L^{1}(\mathbb{R})}\|g\|_{L^{\frac{1}{\theta_1}}(\mathbb{R})}
  \end{align}
for all $j<j_0$ with $\theta_1\in (0,1)$. From Proposition \ref{proposition 3.2}, reiterating this interpolation process, we also have
\begin{align}\label{eq:4.7}
  \left\|H_{\alpha,\gamma,j}(f,g)\chi_{E}\right\|_{L^{\frac{1}{1+\theta_1}}(\mathbb{R})}\lesssim \min\left\{2^{j_0},2^{{j_0}\theta_{1}}|E|^{1-\theta_{1}}\right\}\|f\|_{L^{1}(\mathbb{R})}\|g\|_{L^{\frac{1}{\theta_1}}(\mathbb{R})}
  \end{align}
for all $j=j_0, j_0+1, j_0+2$. Using the interpolation argument on \eqref{eq:4.2} with \eqref{eq:4.4}, we claim that
\begin{align}\label{eq:4.6}
  \left\|H_{\alpha,\gamma,j}(f,g)\chi_{E}\right\|_{L^{\frac{1}{1+\theta_1}}(\mathbb{R})}\lesssim \min\left\{2^j,\frac{2^{j(1-\theta_{1})}|E|^{\theta_1}}{|\gamma'(2^{j-1})|^{\theta_{1}}},2^{j\theta_1} |E|^{1-\theta_1},\frac{|E|}{|\gamma'(2^{j-1})|^{\theta_1}}\right\}\|f\|_{L^{1}(\mathbb{R})}\|g\|_{L^{\frac{1}{\theta_1}}(\mathbb{R})}
  \end{align}
for all $j>j_0+2$.

Now pick $E_{\lambda}:=\{x\in{\mathbb{R}}:\ |H_{\alpha,\gamma}(f,g)(x)|>\lambda\}$ with nonnegative functions $f, g \in L^1(\mathbb{R}) \cap L^{\infty}(\mathbb{R})$. One may use Chebyshev's inequality and \eqref{eq:4.1}, \eqref{eq:4.5}-\eqref{eq:4.6} to deduce that
\begin{align}\label{eq:4.8}
 \lambda^{\frac{1}{1+\theta_{1}}}|E_{\lambda}|&\lesssim \int_{E_{\lambda}}\left(\sum\limits_{j\in\mathbb{Z}}2^{(\alpha-1)j}H_{\alpha,\gamma,j}(f,g)(x)\right)^{\frac{1}{1+\theta_{1}}}\,\textrm{d}x
\lesssim\sum\limits_{j\in\mathbb{Z}}2^{{\frac{\alpha-1}{1+\theta_{1}}}j}\int_{E_{\lambda}}H_{\alpha,\gamma,j}(f,g)(x))^{\frac{1}{1+\theta_{1}}}\,\textrm{d}x\\ \nonumber
&\lesssim \Delta_{1}+\Delta_{2}+\Delta_{3},
  \end{align}
where
$$\begin{cases}
   \Delta_{1}:=\sum\limits_{j<j_{0}}2^{{\frac{\alpha-1}{1+\theta_{1}}j}}\min\left\{2^j,|E_\lambda|\right\}^{\frac{1}{1+\theta_{1}}}
\|f\|_{L^{1}(\mathbb{R})}^{\frac{1}{1+\theta_{1}}}\|g\|_{L^{\frac{1}{\theta_1}}(\mathbb{R})}^{\frac{1}{1+\theta_{1}}}; \\
   \Delta_{2}:=2^{{\frac{\alpha-1}{1+\theta_{1}}}j_{0}}\min\left\{2^{j_0},2^{{j_0}\theta_{1}}|E_\lambda|^{1-\theta_{1}}\right\}^{\frac{1}{1+\theta_{1}}}\|f\|_{L^{1}(\mathbb{R})}^{\frac{1}{1+\theta_{1}}}\|g\|_{L^{\frac{1}{\theta_1}}(\mathbb{R})}^{\frac{1}{1+\theta_{1}}};\\
 \Delta_{3}:=\sum\limits_{j>j_{0}+2}2^{{\frac{\alpha-1}{1+\theta_{1}}}j}\min\left\{2^j,\frac{2^{j(1-\theta_{1})}|E_\lambda|^{\theta_1}}{|\gamma'(2^{j-1})|^{\theta_{1}}},2^{j\theta_1} |E_\lambda|^{1-\theta_1},\frac{|E_\lambda|}{|\gamma'(2^{j-1})|^{\theta_1}}\right\}^{\frac{1}{1+\theta_{1}}}\|f\|_{L^{1}(\mathbb{R})}^{\frac{1}{1+\theta_{1}}}\|g\|_{L^{\frac{1}{\theta_1}}(\mathbb{R})}^{\frac{1}{1+\theta_{1}}}.
 \end{cases}$$

These three terms are treated separately in the following section. We are ready to calculate $\Delta_{1}$, which can be written as
\begin{align*}
\left[\sum\limits_{j<j_{0},~{|E_{\lambda}|\geq2^j}}2^{{\frac{\alpha}{1+\theta_{1}}}j}+
\sum\limits_{j<j_{0},~{|E_{\lambda}|<2^j}}2^{{\frac{\alpha-1}{1+\theta_{1}}}j}|E_\lambda|^{\frac{1}{1+\theta_{1}}}\right]
\|f\|_{L^{1}(\mathbb{R})}^{\frac{1}{1+\theta_{1}}}\|g\|_{L^{\frac{1}{\theta_1}}(\mathbb{R})}^{\frac{1}{1+\theta_{1}}}.
\end{align*}
Notice that
$$\sum\limits_{j<j_{0},~{E_{\lambda}\geq2^j}}2^{{\frac{\alpha}{1+\theta_{1}}}j}\lesssim
\begin{cases}
 2^{{\frac{\alpha}{1+\theta_{1}}}{j_0}} , & |E_\lambda|\geq2^{j_0}; \\
  |E_\lambda|^{\frac{\alpha}{1+\theta_{1}}}, & |E_\lambda|<2^{j_0},
\end{cases}$$
and
$$\sum\limits_{j<j_{0},~{|E_{\lambda}|<2^j}}2^{{\frac{\alpha-1}{1+\theta_{1}}}j}|E_\lambda|^{\frac{1}{1+\theta_{1}}}\lesssim
\begin{cases}
 0, & |E_\lambda|\geq2^{j_0}; \\
  |E_\lambda|^{\frac{\alpha}{1+\theta_{1}}}, & |E_\lambda|<2^{j_0}.
\end{cases}$$
Hence we have the estimate
\begin{align}\label{eq:4.9}
\Delta_{1}\lesssim
\begin{cases}
  \|f\|_{L^{1}(\mathbb{R})}^{\frac{1}{1+\theta_{1}}}\|g\|_{L^{\frac{1}{\theta_1}}(\mathbb{R})}^{\frac{1}{1+\theta_{1}}}, & |E_\lambda|\geq2^{j_0}; \\
  |E_\lambda|^{\frac{\alpha}{1+\theta_{1}}}\|f\|_{L^{1}(\mathbb{R})}^{\frac{1}{1+\theta_{1}}}\|g\|_{L^{\frac{1}{\theta_1}}(\mathbb{R})}^{\frac{1}{1+\theta_{1}}}, & |E_\lambda|<2^{j_0}.
\end{cases}
\end{align}

Now we turn to the second term $\Delta_{2}$, it is easy to see that
\begin{align}\label{eq:4.10}
\Delta_{2}\lesssim
\begin{cases}
   \|f\|_{L^{1}(\mathbb{R})}^{\frac{1}{1+\theta_{1}}}\|g\|_{L^{\frac{1}{\theta_1}}(\mathbb{R})}^{\frac{1}{1+\theta_{1}}}, &  |E_\lambda|\geq2^{j_0}; \\
   |E_\lambda|^{\frac{1-\theta_{1}}{1+\theta_{1}}}\|f\|_{L^{1}(\mathbb{R})}^{\frac{1}{1+\theta_{1}}}\|g\|_{L^{\frac{1}{\theta_1}}(\mathbb{R})}^{\frac{1}{1+\theta_{1}}}, & |E_\lambda|<2^{j_0}.
 \end{cases}
\end{align}

Our next goal is to consider $\Delta_{3}$, a calculation gives
\begin{align*}
\min\left\{2^j,\frac{2^{j(1-\theta_{1})}|E_\lambda|^{\theta_1}}{|\gamma'(2^{j-1})|^{\theta_{1}}},2^{j\theta_1} |E_\lambda|^{1-\theta_1},\frac{|E_\lambda|}{|\gamma'(2^{j-1})|^{\theta_1}}\right\}=
 \begin{cases}
   2^j, &  |E_\lambda|\geq2^j|\gamma'(2^{j-1})|; \\
   \frac{2^{j(1-\theta_{1})}|E_\lambda|^{\theta_1}}{|\gamma'(2^{j-1})|^{\theta_{1}}},&2^j\leq|E_\lambda|<2^j|\gamma'(2^{j-1})|;\\
  \frac{|E_\lambda|}{|\gamma'(2^{j-1})|^{\theta_1}} , & |E_\lambda|<2^j.
 \end{cases}
 \end{align*}
We can then write $\Delta_{3}=\Delta_{3,1}+\Delta_{3,2}+\Delta_{3,3}$, where
$$\begin{cases}
   \Delta_{3,1}:= \sum\limits_{  \substack{j>j_{0}+2   \\|E_\lambda|\geq2^j|\gamma'(2^{j-1})|  }}    2^{{\frac{\alpha}{1+\theta_{1}}}j}
\|f\|_{L^{1}(\mathbb{R})}^{\frac{1}{1+\theta_{1}}}\|g\|_{L^{\frac{1}{\theta_1}}(\mathbb{R})}^{\frac{1}{1+\theta_{1}}}; \\
   \Delta_{3,2}:=\sum\limits_{  \substack{j>j_{0}+2   \\2^j\leq|E_\lambda|<2^j|\gamma'(2^{j-1})|  }}   {\frac{2^{{\frac{\alpha-\theta_{1}}{1+\theta_{1}}}j}|E_\lambda|^{\frac{\theta_{1}}{1+\theta_{1}}}}{|\gamma'(2^{j-1})|^{\frac{\theta_{1}}{1+\theta_{1}}}}} \|f\|_{L^{1}(\mathbb{R})}^{\frac{1}{1+\theta_{1}}}\|g\|_{L^{\frac{1}{\theta_1}}(\mathbb{R})}^{\frac{1}{1+\theta_{1}}};\\
 \Delta_{3,3}:=\sum\limits_{  \substack{j>j_{0}+2   \\|E_\lambda|<2^j}}   {\frac{2^{{\frac{\alpha-1}{1+\theta_{1}}}j}|E_\lambda|^{\frac{1}{1+\theta_{1}}}}{|\gamma'(2^{j-1})|^{\frac{\theta_{1}}{1+\theta_{1}}}}}
 \|f\|_{L^{1}(\mathbb{R})}^{\frac{1}{1+\theta_{1}}}\|g\|_{L^{\frac{1}{\theta_1}}(\mathbb{R})}^{\frac{1}{1+\theta_{1}}}.
 \end{cases}$$
Corresponding to $\Delta_{3,1}$, for any $\varepsilon>0$, it follows from Lemma \ref{lemma 2.1} that
$$\sum\limits_{  \substack{j>j_{0}+2   \\|E_\lambda|\geq2^j|\gamma'(2^{j-1})|  }}    2^{{\frac{\alpha}{1+\theta_{1}}}j}\lesssim
\begin{cases}
   |E_\lambda|^{\frac{\alpha}{(\omega_{\infty,1}-\varepsilon)(1+\theta_{1})}}, & |E_\lambda|\geq2^{j_0+2}|\gamma'(2^{j_0+1})|; \\
  0, & 2^{j_0+2}\leq|E_{\lambda}|<2^{j_0+2}|\gamma'(2^{j_{0}+1})|;\\
  0, &|E_{\lambda}|<2^{j_0+2}.
\end{cases}$$
On the other hand, let us set $\varepsilon\in(0,\omega_{\infty,1})$, and we take that $\alpha<\theta_{1}(\omega_{\infty,1}-\varepsilon)$, then by Lemma \ref{lemma 2.1} one has
$$\sum\limits_{  \substack{j>j_{0}+2   \\2^j\leq|E_\lambda|<2^j|\gamma'(2^{j-1})|  }}   {\frac{2^{{\frac{\alpha-\theta_{1}}{1+\theta_{1}}}j}|E_\lambda|^{\frac{\theta_{1}}{1+\theta_{1}}}}{|\gamma'(2^{j-1})|^{\frac{\theta_{1}}{1+\theta_{1}}}}} \lesssim
\begin{cases}
  |E_\lambda|^{{\frac{\alpha+\theta_1(\omega_{\infty,2}-\omega_{\infty,1}+2\varepsilon)}{(\omega_{\infty,2}+\varepsilon)(1+\theta_{1})}}}, & |E_\lambda|\geq2^{j_0+2}|\gamma'(2^{j_{0}+1})|;\\
  |E_\lambda|^{{\frac{\alpha+\theta_1(\omega_{\infty,2}-\omega_{\infty,1}+2\varepsilon)}{(\omega_{\infty,2}+\varepsilon)(1+\theta_{1})}}}, & 2^{j_0+2}\leq|E_{\lambda}|<2^{j_0+2}|\gamma'(2^{j_{0}+1})|;\\
 0, &|E_{\lambda}|<2^{j_0+2}.
\end{cases}$$
This is the estimate about $\Delta_{3,2}$. Furthermore, for estimating $\Delta_{3,3}$, we need to take $\alpha-1+\theta_{1}>0$. It follows from $\omega_{\infty,1}>1$ that $\alpha-1+\theta_{1}-\theta_{1}(\omega_{\infty,1}-\varepsilon)<0$ for all $\varepsilon\in(0,\omega_{\infty,1}-1]$. From Lemma \ref{lemma 2.1}, we then have
$$\sum\limits_{  \substack{j>j_{0}+2   \\|E_\lambda|<2^j}}  {\frac{2^{{\frac{\alpha-1}{1+\theta_{1}}}j}|E_\lambda|^{\frac{1}{1+\theta_{1}}}}{|\gamma'(2^{j-1})|^{\frac{\theta_{1}}{1+\theta_{1}}}}}\lesssim
\begin{cases}
  |E_\lambda|^{{\frac{\alpha-\theta_1(\omega_{\infty,1}-1-\varepsilon)}{1+\theta_{1}}}},
  & |E_\lambda|\geq2^{j_0+2}|\gamma'(2^{j_{0}+1})|; \\
  |E_\lambda|^{{\frac{\alpha-\theta_1(\omega_{\infty,1}-1-\varepsilon)}{1+\theta_{1}}}},  & 2^{j_0+2}\leq|E_{\lambda}|<2^{j_0+2}|\gamma'(2^{j_{0}+1})|;\\
  |E_\lambda|^{{\frac{1}{1+\theta_1}}}, &|E_{\lambda}|<2^{j_0+2}.
\end{cases}$$
Notice that  $$ |E_\lambda|^{{\frac{\alpha-\theta_1(\omega_{\infty,1}-1-\varepsilon)}{1+\theta_{1}}}}\leq|E_\lambda|^{\frac{\alpha}{(\omega_{\infty,1}-\varepsilon)(1+\theta_{1})}} \leq  |E_\lambda|^{{\frac{\alpha+\theta_1(\omega_{\infty,2}-\omega_{\infty,1}+2\varepsilon)}{(\omega_{\infty,2}+\varepsilon)(1+\theta_{1})}}}$$
holds for $|E_\lambda|\geq2^{j_0+2}$ if $\varepsilon\in(0,\min\{\omega_{\infty,1}, \omega_{\infty,1}-1\})$. Consequently,
\begin{align}\label{eq:4.11}
\Delta_{3}\lesssim
\begin{cases}
   |E_\lambda|^{{\frac{\alpha+\theta_1(\omega_{\infty,2}-\omega_{\infty,1}+2\varepsilon)}{(\omega_{\infty,2}+\varepsilon)(1+\theta_{1})}}}
 \|f\|_{L^{1}(\mathbb{R})}^{\frac{1}{1+\theta_{1}}}\|g\|_{L^{\frac{1}{\theta_1}}(\mathbb{R})}^{\frac{1}{1+\theta_{1}}}, & |E_\lambda|\geq2^{j_0+2};\\
  |E_\lambda|^{{\frac{1}{1+\theta_{1}}}}
\|f\|_{L^{1}(\mathbb{R})}^{\frac{1}{1+\theta_{1}}}\|g\|_{L^{\frac{1}{\theta_1}}(\mathbb{R})}^{\frac{1}{1+\theta_{1}}}, &|E_{\lambda}|<2^{j_0+2}.
\end{cases}
\end{align}

It follows from $\alpha-1+\theta_{1}>0$ and $\theta_1\in (0,1)$ that
$ |E_\lambda|^{{\frac{1}{1+\theta_{1}}}}\leq |E_\lambda|^{\frac{\alpha}{1+\theta_{1}}}\leq |E_\lambda|^{\frac{1-\theta_{1}}{1+\theta_{1}}}$
holds for $|E_{\lambda}|<2^{j_0+2}$. This, combined with \eqref{eq:4.8}-\eqref{eq:4.11}, we conclude that
\begin{align}\label{eq:4.12}
\lambda^{\frac{1}{1+\theta_{1}}}|E_{\lambda}|\lesssim
\begin{cases}
   |E_\lambda|^{\frac{\alpha+\theta_1(\omega_{\infty,2}-\omega_{\infty,1}+2\varepsilon)}{(\omega_{\infty,2}+\varepsilon)(1+\theta_{1})}}
 \|f\|_{L^{1}(\mathbb{R})}^{\frac{1}{1+\theta_{1}}}\|g\|_{L^{\frac{1}{\theta_1}}(\mathbb{R})}^{\frac{1}{1+\theta_{1}}}, & |E_\lambda|\geq2^{j_0+2}; \\
 |E_\lambda|^{\frac{1-\theta_{1}}{1+\theta_{1}}}
\|f\|_{L^{1}(\mathbb{R})}^{\frac{1}{1+\theta_{1}}}\|g\|_{L^{\frac{1}{\theta_1}}(\mathbb{R})}^{\frac{1}{1+\theta_{1}}}, &|E_{\lambda}|<2^{j_0+2}.
\end{cases}
\end{align}
By a calculation, for $\varepsilon$ small enough, we have that $\frac{\omega_{\infty,2}-\omega_{\infty,1}}{\omega_{\infty,2}-1}<\frac{\alpha}{1-\alpha}<\frac{\omega_{\infty,2}}{\omega_{\infty,2}-\omega_{\infty,1}+1}$ implies
\begin{align}\label{eq:4.13}
\max\left\{1-\alpha, \alpha, \frac{\alpha}{\omega_{\infty,1}-\varepsilon}\right\}<\frac{\omega_{\infty,2}-\alpha+\varepsilon}{2\omega_{\infty,2}-\omega_{\infty,1}+3\varepsilon}.
\end{align}
Then, we can take
\begin{align}\label{eq:4.14}
\theta_1\leq \frac{\omega_{\infty,2}-\alpha+\varepsilon}{2\omega_{\infty,2}-\omega_{\infty,1}+3\varepsilon}
\end{align}
with $\varepsilon$ small enough, which further implies that
$$\frac{\alpha+\theta_1(\omega_{\infty,2}-\omega_{\infty,1}+2\varepsilon)}{(\omega_{\infty,2}+\varepsilon)(1+\theta_{1})}\leq \frac{1-\theta_{1}}{1+\theta_{1}},$$
This, together with \eqref{eq:4.12}, gives
\begin{align}\label{eq:4.15}
\lambda|E_{\lambda}|^{\frac{\omega_{\infty,2}+ \theta_1(\omega_{\infty,1}-\varepsilon)-\alpha+\varepsilon}{\omega_{\infty,2}+\varepsilon}}\lesssim \|f\|_{L^{1}(\mathbb{R})}\|g\|_{L^{\frac{1}{\theta_1}}(\mathbb{R})}
\end{align}
for all $0<|E_\lambda|<\infty$, where $\varepsilon$ small enough and $\theta_1$ satisfies \eqref{eq:4.14}. Therefore, we have established the $L^{1}(\mathbb{R})\times L^{\frac{1}{\theta_{1}}}(\mathbb{R})\rightarrow L^{\frac{\omega_{\infty,2}+\varepsilon}{\omega_{\infty,2}+ \theta_1(\omega_{\infty,1}-\varepsilon)-\alpha+\varepsilon}, \infty}(\mathbb{R})$ boundedness for $H_{\alpha,\gamma}$.

It is easy to see that $\frac{\omega_{\infty,2}+ \theta_1(\omega_{\infty,1}-\varepsilon)-\alpha+\varepsilon}{\omega_{\infty,2}+\varepsilon}$ takes the maximal value $\frac{2(\omega_{\infty,2}-\alpha+\varepsilon)}{2\omega_{\infty,2}-\omega_{\infty,1}+3\varepsilon}$ when $\theta_1= \frac{\omega_{\infty,2}-\alpha+\varepsilon}{2\omega_{\infty,2}-\omega_{\infty,1}+3\varepsilon}$. On the other hand, from \eqref{eq:4.13}, it follows $\frac{\alpha}{\omega_{\infty,1}-\varepsilon}<\frac{\omega_{\infty,2}-\alpha+\varepsilon}{2\omega_{\infty,2}-\omega_{\infty,1}+3\varepsilon}$, which equals to $2\alpha+\varepsilon<\omega_{\infty,1}$, we then have $2\alpha<\omega_{\infty,1}$. This, together with $\alpha<1<\omega_{\infty,1}\leq\omega_{\infty,2}$, shows $3\alpha<\omega_{\infty,1}+\omega_{\infty,2}$. Therefore,
$$\frac{\omega_{\infty,2}-\alpha+\varepsilon}{2\omega_{\infty,2}-\omega_{\infty,1}+3\varepsilon}< \frac{\omega_{\infty,2}-\alpha}{2\omega_{\infty,2}-\omega_{\infty,1}}$$
holds for all $\varepsilon$ small enough. Putting things together, we conclude that $H_{\alpha,\gamma}$ maps $L^{1}(\mathbb{R})\times L^{2r( \varepsilon )   }(\mathbb{R})$ to $L^{r(\varepsilon ), \infty}(\mathbb{R})$, where $\frac{1}{r( \varepsilon ) }:=\frac{2(\omega_{\infty,2}-\alpha)}{2\omega_{\infty,2}-\omega_{\infty,1}}-\varepsilon $ with $\varepsilon$ small enough. Therefore, we finish the proof of Proposition \ref{proposition 4.1}.
\end{proof}

\begin{remark}\label{remark 4.2}
In fact, \eqref{eq:4.13} is also used to prove the following Proposition \ref{proposition 5.1}. In the proof of Proposition \ref{proposition 4.1}, it suffices to replace \eqref{eq:4.13} by
\begin{align*}
\max\left\{1-\alpha, \frac{\alpha}{\omega_{\infty,1}-\varepsilon}\right\}<\frac{\omega_{\infty,2}-\alpha+\varepsilon}{2\omega_{\infty,2}-\omega_{\infty,1}+3\varepsilon},
\end{align*}
which leads to the condition $\frac{\alpha}{1-\alpha}<\frac{\omega_{\infty,2}}{\omega_{\infty,2}-\omega_{\infty,1}+1}$ in Theorem \ref{maintheorem} can be replaced as $2\alpha<\omega_{\infty,1}$.
\end{remark}

\begin{proposition}\label{proposition 4.2}
Let $\gamma $ be the same as in Theorem \ref{maintheorem2}. Then, there exists a positive constant $C$, such that
$$\left\|H_{\alpha,\gamma}(f,g)\right\|_{L^{\bar{r}(\varepsilon),\infty}(\mathbb{R})}\leq C \|f\|_{L^{1}(\mathbb{R})}\left\|g\right\|_{L^{\bar{q}(\varepsilon)}(\mathbb{R})},$$
where
\begin{align*}
\begin{cases}
\frac{1}{\bar{r}(\varepsilon )}:= \frac{2(1-\alpha)}{2-\omega_{\infty,1}}-\varepsilon ~~\textrm{and}~~\frac{1}{\bar{q}(\varepsilon )}:=\frac{1}{2\bar{r}(\varepsilon )}   ,&\omega_{\infty,1}\leq2\alpha;\\
\frac{1}{\bar{r}(\varepsilon )}:= 1-\varepsilon  ~~\textrm{and}~~\frac{1}{\bar{q}(\varepsilon )}\in \left[\frac{\alpha}{\omega_{\infty,1}}-\varepsilon, \frac{1}{2}-\varepsilon\right] ,& 2\alpha<\omega_{\infty,1},
\end{cases}
\end{align*}
with $\varepsilon>0$ small enough.
\end{proposition}

\begin{proof}
The main difference here lies in the establishing of $\Delta_{3,2}$. In fact, let $\varepsilon\in(0,\omega_{\infty,1})$, and we take $\alpha>\theta_{1}(\omega_{\infty,1}-\varepsilon)$. By Lemma \ref{lemma 2.1}, one immediately has
$$\sum\limits_{  \substack{j>j_{0}+2   \\2^j\leq|E_\lambda|<2^j|\gamma'(2^{j-1})|  }}   {\frac{2^{{\frac{\alpha-\theta_{1}}{1+\theta_{1}}}j}|E_\lambda|^{\frac{\theta_{1}}{1+\theta_{1}}}}{|\gamma'(2^{j-1})|^{\frac{\theta_{1}}{1+\theta_{1}}}}} \lesssim
\begin{cases}
  |E_\lambda|^{{\frac{\alpha-\theta_1(\omega_{\infty,1}-1-\varepsilon)}{1+\theta_1}}}, & |E_\lambda|\geq2^{j_0+2}|\gamma'(2^{j_{0}+1})|;\\
   |E_\lambda|^{{\frac{\alpha-\theta_1(\omega_{\infty,1}-1-\varepsilon)}{1+\theta_1}}}, & 2^{j_0+2}\leq|E_{\lambda}|<2^{j_0+2}|\gamma'(2^{j_{0}+1})|;\\
 0, &|E_{\lambda}|<2^{j_0+2}.
\end{cases}$$
The estimation of $\Delta_{3,3}$ is the same as in Proposition \ref{proposition 4.1}, and we also need to take $\alpha-1+\theta_{1}>0$, which can be realized by the condition that $\frac{\alpha}{1-\alpha}\geq \omega_{\infty,1}>1$. Notice that $\alpha>\theta_{1}(\omega_{\infty,1}-\varepsilon)$, which implies $|E_\lambda|^{\frac{\alpha}{(\omega_{\infty,1}-\varepsilon)(1+\theta_{1})}}\leq|E_\lambda|^{{\frac{\alpha-\theta_1(\omega_{\infty,1}-1-\varepsilon)}{1+\theta_1}}}$ for $|E_\lambda|\geq2^{j_0+2}$. Thus one can write
\begin{align*}
\Delta_{3}\lesssim
\begin{cases}
  |E_\lambda|^{{\frac{\alpha-\theta_1(\omega_{\infty,1}-1-\varepsilon)}{1+\theta_1}}}
 \|f\|_{L^{1}(\mathbb{R})}^{\frac{1}{1+\theta_{1}}}\|g\|_{L^{\frac{1}{\theta_1}}(\mathbb{R})}^{\frac{1}{1+\theta_{1}}}, & |E_\lambda|\geq2^{j_0+2}; \\
  |E_\lambda|^{{\frac{1}{1+\theta_{1}}}}
\|f\|_{L^{1}(\mathbb{R})}^{\frac{1}{1+\theta_{1}}}\|g\|_{L^{\frac{1}{\theta_1}}(\mathbb{R})}^{\frac{1}{1+\theta_{1}}}, &|E_{\lambda}|<2^{j_0+2}.
\end{cases}
\end{align*}
As in Proposition \ref{proposition 4.1}, we have an estimate
\begin{align}\label{eq:4.19}
\lambda^{\frac{1}{1+\theta_{1}}}|E_{\lambda}|\lesssim
\begin{cases}
   |E_\lambda|^{{\frac{\alpha-\theta_1(\omega_{\infty,1}-1-\varepsilon)}{1+\theta_1}}}
 \|f\|_{L^{1}(\mathbb{R})}^{\frac{1}{1+\theta_{1}}}\|g\|_{L^{\frac{1}{\theta_1}}(\mathbb{R})}^{\frac{1}{1+\theta_{1}}}, & |E_\lambda|\geq2^{j_0+2}; \\
 |E_\lambda|^{\frac{1-\theta_{1}}{1+\theta_{1}}}
\|f\|_{L^{1}(\mathbb{R})}^{\frac{1}{1+\theta_{1}}}\|g\|_{L^{\frac{1}{\theta_1}}(\mathbb{R})}^{\frac{1}{1+\theta_{1}}}, &|E_{\lambda}|<2^{j_0+2}.
\end{cases}
\end{align}

We now want to obtain that
\begin{align}\label{eq:4.20}
\frac{\alpha-\theta_1(\omega_{\infty,1}-1-\varepsilon)}{1+\theta_1}\leq \frac{1-\theta_{1}}{1+\theta_{1}},
\end{align}
which equals to $\theta_1(2-\omega_{\infty,1}+\varepsilon)\leq1-\alpha$. If $2-\omega_{\infty,1}+\varepsilon\leq0$, \eqref{eq:4.20} is obvious. If $2-\omega_{\infty,1}+\varepsilon>0$, \eqref{eq:4.20} equals to
\begin{align}\label{eq:4.21}
\theta_1\leq\frac{1-\alpha}{2-\omega_{\infty,1}+\varepsilon}.
\end{align}
Clearly $1-\alpha\leq\frac{1-\alpha}{2-\omega_{\infty,1}+\varepsilon}$ holds for all $\varepsilon\in(0,\omega_{\infty,1}-1]$, which shows \eqref{eq:4.21} can be realized. This, combined with \eqref{eq:4.19} and \eqref{eq:4.20}, means that
\begin{align*}
\lambda|E_{\lambda}|^{1-\alpha+\theta_1(\omega_{\infty,1}-\varepsilon)}\lesssim\|f\|_{L^{1}(\mathbb{R})}\|g\|_{L^{\frac{1}{\theta_1}}(\mathbb{R})}
\end{align*}
for all $0<|E_\lambda|<\infty$, where
\begin{align}\label{eq:4.22}
\begin{cases}
\theta_1\leq\frac{1-\alpha}{2-\omega_{\infty,1}+\varepsilon}, &\omega_{\infty,1}<2\alpha+\varepsilon;\\
 \theta_1< \frac{\alpha}{\omega_{\infty,1}-\varepsilon}, & 2\alpha+\varepsilon\leq\omega_{\infty,1},
\end{cases}
\end{align}
and $\varepsilon>0$ small enough. Therefore, we have established the $L^{1}(\mathbb{R})\times L^{\frac{1}{\theta_{1}}}(\mathbb{R})\rightarrow L^{\frac{1}{1-\alpha+\theta_1(\omega_{\infty,1}-\varepsilon)}, \infty}(\mathbb{R})$ boundedness for $H_{\alpha,\gamma}$.

On the other hand, by \eqref{eq:4.22}, for $\varepsilon>0$ small enough, one easily sees that
\begin{align*}
\begin{cases}
1-\alpha+\theta_1(\omega_{\infty,1}-\varepsilon)\leq\frac{2(1-\alpha)}{2-\omega_{\infty,1}+\varepsilon}<\frac{2(1-\alpha)}{2-\omega_{\infty,1}}, &\omega_{\infty,1}<2\alpha+\varepsilon;\\
 1-\alpha+\theta_1(\omega_{\infty,1}-\varepsilon)< 1, & 2\alpha+\varepsilon\leq\omega_{\infty,1}.
\end{cases}
\end{align*}
Therefore, we conclude that $H_{\alpha,\gamma}$ maps $L^{1}(\mathbb{R})\times L^{\bar{q}( \varepsilon )   }(\mathbb{R})$ to $L^{\bar{r}(\varepsilon ), \infty}(\mathbb{R})$ with $\varepsilon>0$ small enough. Here, $\frac{1}{\bar{r}( \varepsilon )}$ tends to maximal value as $\varepsilon\rightarrow 0$.

Notice that $\frac{\alpha}{\omega_{\infty,1}-\varepsilon}$ increases as $\varepsilon$ increases when $2\alpha+\varepsilon\leq\omega_{\infty,1}$, then $\frac{1}{\bar{q}( \varepsilon )}$ doesn't tend to maximal value as $\varepsilon\rightarrow 0$. We now want $\frac{1}{\bar{q}( \varepsilon )}$ to take maximal value. In fact, from the process of our proofs, we can find that $\varepsilon$ need only satisfies $\varepsilon\leq\omega_{\infty,1}-2\alpha$ and $\varepsilon\leq\omega_{\infty,1}-1$ when $\omega_{\infty,1}>2\alpha$. Moreover, using $2\alpha<\omega_{\infty,1}\leq \frac{\alpha}{1-\alpha}$, we have $\frac{1}{2}<\alpha$. Therefore, we should only to take $\varepsilon\leq\omega_{\infty,1}-2\alpha$. It is easy to check that $\frac{1}{\bar{q}( \varepsilon )}$ tends to the maximal value $\frac{1}{2}$ when $\varepsilon=\omega_{\infty,1}-2\alpha$, and $\frac{1}{\bar{r}(\varepsilon )}$ tends to $1$. Then, $H_{\alpha,\gamma}$ maps $L^{1}(\mathbb{R})\times L^{\frac{2}{1-\varepsilon}}(\mathbb{R})$ to $L^{\frac{1}{1-\varepsilon}, \infty}(\mathbb{R})$ with $\varepsilon>0$ small enough when $\omega_{\infty,1}>2\alpha$.

Putting things together, we finish the proof of Proposition \ref{proposition 4.2}.
\end{proof}

\section{Estimates for $H_{\alpha,\gamma}$ with $g\in L^1(\mathbb{R})$}

In this section, we will establish some $L^p(\mathbb{R})\times L^q(\mathbb{R}) \rightarrow L^{r,\infty}(\mathbb{R})$ estimates for $H_{\alpha,\gamma}$ with $q=1$ by considering the corresponding estimates for $H_{\alpha,\gamma,j}$.

\begin{proposition}\label{proposition 5.1}
Let $\gamma $ be the same as in Theorem \ref{maintheorem}. Then, there exists a positive constant $C$, such that
$$\left\|H_{\alpha,\gamma}(f,g)\right\|_{L^{r(\varepsilon),\infty}(\mathbb{R})}\leq C \|f\|_{L^{2r(\varepsilon)}(\mathbb{R})}\left\|g\right\|_{L^{1}(\mathbb{R})},$$
where $\frac{1}{r( \varepsilon ) }=\frac{2(\omega_{\infty,2}-\alpha)}{2\omega_{\infty,2}-\omega_{\infty,1}}-\varepsilon $ with $\varepsilon>0$ small enough as in Proposition \ref{proposition 4.1}.
\end{proposition}

\begin{proof}
Let $E$ be a measurable set with finite measure, and recall that $H_{\alpha,\gamma,j}$ defined in \eqref{eq:3.1} with $j\in \mathbb{Z}$, by the fact that $|\gamma'(t)|$ is strictly increasing on $t\in\mathbb{R}^+$, a routine calculation gives
\begin{align*}
  \left\|H_{\alpha,\gamma,j}(f,g)\chi_{E}\right\|_{L^{1}(\mathbb{R})}
\leq\|f\|_{L^{\infty}(\mathbb{R})}\int_E\int_{2^{j-1}}^{2^j}g(x-\gamma(t))\,\textrm{d}t\,\textrm{d}x\leq \frac{|E|}{|\gamma'(2^{j-1})|} \|f\|_{L^{\infty}(\mathbb{R})}\|g\|_{L^{1}(\mathbb{R})}.
 \end{align*}
Combining Propositions \ref{proposition 3.1} and \ref{proposition 3.2}, for any $j\in \mathbb{Z}$, we immediately have
\begin{align}\label{eq:5.1}
  \left\|H_{\alpha,\gamma,j}(f,g)\chi_{E}\right\|_{L^{1}(\mathbb{R})}
\leq \min\left\{2^j,  \frac{|E|}{|\gamma'(2^{j-1})|} \right\} \|f\|_{L^{\infty}(\mathbb{R})}\|g\|_{L^{1}(\mathbb{R})}.
 \end{align}
Notice that $H_{\alpha,\gamma,j}(f,g)\chi_{E}$ can be viewed as a linear operator about $f$ for any fixed $g\in L^1(\mathbb{R})$, by interpolation between \eqref{eq:5.1} with \eqref{eq:4.3}, we conclude that
 \begin{align*}
  \left\|H_{\alpha,\gamma,j}(f,g)\chi_{E}\right\|_{L^{\frac{1}{1+\theta_2}}(\mathbb{R})} \lesssim \min\left\{2^j,2^{j(1-\theta_{2})}|E|^{\theta_{2}},{\frac{2^{j\theta_2}|E|^{1-\theta_{2}}}{|\gamma'(2^{j-1})|^{1-\theta_{2}}}} ,{\frac{|E|}{|\gamma'(2^{j-1})|^{1-\theta_{2}}}}\right\}\|f\|_{L^{\frac{1}{\theta_2}}(\mathbb{R})}\|g\|_{L^{1}(\mathbb{R})}
 \end{align*}
for all $j<j_0$ with $\theta_2\in (0,1)$. From Proposition \ref{proposition 3.2}, reiterating this interpolation process, then the following estimate is valid:
\begin{align*}
  \left\|H_{\alpha,\gamma,j}(f,g)\chi_{E}\right\|_{L^{\frac{1}{1+\theta_1}}(\mathbb{R})}\lesssim \min\left\{2^{j_0},\frac{2^{{j_0}\theta_{2}}|E|^{1-\theta_{2}}}{ |\gamma'(2^{j_0-1})|^{1-\theta_{2}} }\right\}\|f\|_{L^{\frac{1}{\theta_2}}(\mathbb{R})}\|g\|_{L^{1}(\mathbb{R})}
  \end{align*}
for all $j=j_0, j_0+1, j_0+2$. A further interpolation between \eqref{eq:5.1} with \eqref{eq:4.4} gives
\begin{align}\label{eq:5.a}
  \left\|H_{\alpha,\gamma,j}(f,g)\chi_{E}\right\|_{L^{\frac{1}{1+\theta_2}}(\mathbb{R})}
\leq \min\left\{2^j,  \frac{|E|}{|\gamma'(2^{j-1})|} \right\} \|f\|_{L^{\frac{1}{\theta_2}}(\mathbb{R})}\|g\|_{L^{1}(\mathbb{R})}
 \end{align}
for all $j>j_0+2$. Moreover, in order to establish the desired estimate, we bound
\begin{align}\label{eq:5.b}
  \left\|H_{\alpha,\gamma,j}(f,g)\chi_{E}\right\|_{L^{\frac{1}{1+\theta_2}}(\mathbb{R})}
\leq \min\left\{2^j,  \frac{2^{j(1-\theta_2)}|E|^{\theta_2}}{|\gamma'(2^{j-1})|^{\theta_2}} \right\} \|f\|_{L^{\frac{1}{\theta_2}}(\mathbb{R})}\|g\|_{L^{1}(\mathbb{R})}
 \end{align}
for all $j>j_0+2$.

Recall that $E_{\lambda}=\{x\in{\mathbb{R}}:\ |H_{\alpha,\gamma}(f,g)(x)|>\lambda\}$. Then, by a similar argument as in \eqref{eq:4.8}, we conclude that
\begin{align}\label{eq:5.2}
 \lambda^{\frac{1}{1+\theta_2}}|E_{\lambda}|\lesssim \Upsilon_{1}+\Upsilon_{2}+\Upsilon_{3},
  \end{align}
where
$$\begin{cases}
   \Upsilon_{1}:=\sum\limits_{j<j_{0}}2^{{\frac{\alpha-1}{1+\theta_{2}}j}}\min\left\{2^j,2^{j(1-\theta_{2})}|E_{\lambda}|^{\theta_{2}},{\frac{2^{j\theta_2}|E_{\lambda}|^{1-\theta_{2}}}{|\gamma'(2^{j-1})|^{1-\theta_{2}}}} ,{\frac{|E_{\lambda}|}{|\gamma'(2^{j-1})|^{1-\theta_{2}}}}\right\}^{\frac{1}{1+\theta_2}}\|f\|^{\frac{1}{1+\theta_2}}_{L^{\frac{1}{\theta_2}}(\mathbb{R})}\|g\|^{\frac{1}{1+\theta_2}}_{L^{1}(\mathbb{R})}; \\
  \Upsilon_{2}:=2^{{\frac{\alpha-1}{1+\theta_{2}}}j_{0}}\min\left\{2^{j_0},\frac{2^{{j_0}\theta_{2}}|E_{\lambda}|^{1-\theta_{2}}}{ |\gamma'(2^{j_0-1})|^{1-\theta_{2}} }\right\}^{\frac{1}{1+\theta_2}}\|f\|^{\frac{1}{1+\theta_2}}_{L^{\frac{1}{\theta_2}}(\mathbb{R})}\|g\|^{\frac{1}{1+\theta_2}}_{L^{1}(\mathbb{R})};\\
 \Upsilon_{3}:=\sum\limits_{j>j_{0}+2}2^{{\frac{\alpha-1}{1+\theta_{2}}}j}\min\left\{2^j,  \frac{2^{j(1-\theta_2)}|E_{\lambda}|^{\theta_2}}{|\gamma'(2^{j-1})|^{\theta_2}} \right\}^{\frac{1}{1+\theta_2}} \|f\|^{\frac{1}{1+\theta_2}}_{L^{\frac{1}{\theta_2}}(\mathbb{R})}\|g\|^{\frac{1}{1+\theta_2}}_{L^{1}(\mathbb{R})}.
 \end{cases}$$

Consider $\Upsilon_1$, in analogy to the calculation in $\Delta_{3}$, and observe that
\begin{align*}
 \min\left\{2^j,2^{j(1-\theta_{2})}|E_{\lambda}|^{\theta_{2}},{\frac{2^{j\theta_2}|E_{\lambda}|^{1-\theta_{2}}}{|\gamma'(2^{j-1})|^{1-\theta_{2}}}} ,{\frac{|E_{\lambda}|}{|\gamma'(2^{j-1})|^{1-\theta_{2}}}}\right\}\leq
 \begin{cases}
   2^j, &  |E_{\lambda}|\geq 2^j; \\
    2^{j(1-\theta_{2})}|E_{\lambda}|^{\theta_{2}}, &  2^j{|\gamma'(2^{j-1})|}\leq|E_{\lambda}|<  2^j; \\
   {\frac{|E_{\lambda}|}{|\gamma'(2^{j-1})|^{1-\theta_{2}}}}, & |E_{\lambda}|<2^j|\gamma'(2^{j-1})|.
 \end{cases}
 \end{align*}
Hence we can write
\begin{align*}
 \Upsilon_{1}=\Upsilon_{1,1}+\Upsilon_{1,2}+\Upsilon_{1,3}
 \end{align*}
with
$$\begin{cases}
   \Upsilon_{1,1}:=\sum\limits_{  \substack{j<j_{0}   \\ |E_{\lambda}|\geq2^j }}
   2^{{\frac{\alpha}{1+\theta_{2}}j}}\|f\|^{\frac{1}{1+\theta_2}}_{L^{\frac{1}{\theta_2}}(\mathbb{R})}\|g\|^{\frac{1}{1+\theta_2}}_{L^{1}(\mathbb{R})}; \\
  \Upsilon_{1,2}:=\sum\limits_{  \substack{j<j_{0}   \\ 2^j{|\gamma'(2^{j-1})|}\leq|E_{\lambda}|<2^j }} 2^{{\frac{\alpha-\theta_{2}}{1+\theta_{2}}}j}|E_\lambda|^{\frac{\theta_{2}}{1+\theta_{2}}}\|f\|^{\frac{1}{1+\theta_2}}_{L^{\frac{1}{\theta_2}}(\mathbb{R})}\|g\|^{\frac{1}{1+\theta_2}}_{L^{1}(\mathbb{R})}; \\
 \Upsilon_{1,3}:=\sum\limits_{  \substack{j<j_{0}   \\ |E_{\lambda}|<2^j|\gamma'(2^{j-1})| }} \frac{2^{{\frac{\alpha-1}{1+\theta_{2}}j}}|E_\lambda|^{\frac{1}{1+\theta_{2}}}}{|\gamma'(2^{j-1})|^{\frac{1-\theta_{2}}{1+\theta_{2}}}}\|f\|^{\frac{1}{1+\theta_2}}_{L^{\frac{1}{\theta_2}}(\mathbb{R})}\|g\|^{\frac{1}{1+\theta_2}}_{L^{1}(\mathbb{R})}.
 \end{cases}$$
Let's now consider $\Upsilon_{1,1}$. For any $\varepsilon>0$, it follows from Lemma \ref{lemma 2.1} that
$$\sum\limits_{  \substack{j<j_{0}   \\ |E_{\lambda}|\geq2^j }}
   2^{{\frac{\alpha}{1+\theta_{2}}j}}\lesssim
\begin{cases}
  2^{\frac{\alpha}{1+\theta_{2}}j_0}, & |E_\lambda|\geq2^{j_0}; \\
  |E_\lambda|^{\frac{\alpha}{1+\theta_{2}}}, & 2^{j_0}|\gamma'(2^{j_{0}-1})|\leq|E_{\lambda}|<2^{j_0};\\
  |E_\lambda|^{\frac{\alpha}{1+\theta_{2}}}, &|E_{\lambda}|<2^{j_0}|\gamma'(2^{j_{0}-1})|.
\end{cases}$$
To see what happens corresponding to $\Upsilon_{1,2}$ and $\Upsilon_{1,3}$, due to a minor technical issue we will assume that $\alpha<\theta_2$. For any $\varepsilon>0$, by Lemma \ref{lemma 2.1}, it is easy to see that
$$\sum\limits_{  \substack{j<j_{0}   \\ 2^j{|\gamma'(2^{j-1})|}\leq|E_{\lambda}|<2^j }} 2^{{\frac{\alpha-\theta_{2}}{1+\theta_{2}}}j}|E_\lambda|^{\frac{\theta_{2}}{1+\theta_{2}}}\lesssim
\begin{cases}
  0, & |E_\lambda|\geq2^{j_0}; \\
  |E_\lambda|^{\frac{\alpha}{1+\theta_{2}}}, & 2^{j_0}|\gamma'(2^{j_{0}-1})|\leq|E_{\lambda}|<2^{j_0};\\
  |E_\lambda|^{\frac{\alpha}{1+\theta_{2}}}, &|E_{\lambda}|<2^{j_0}|\gamma'(2^{j_{0}-1})|,
\end{cases}$$
and
$$\sum\limits_{  \substack{j<j_{0}   \\ |E_{\lambda}|<2^j|\gamma'(2^{j-1})| }} \frac{2^{{\frac{\alpha-1}{1+\theta_{2}}j}}|E_\lambda|^{\frac{1}{1+\theta_{2}}}}{|\gamma'(2^{j-1})|^{\frac{1-\theta_{2}}{1+\theta_{2}}}}\lesssim
\begin{cases}
  0, & |E_\lambda|\geq2^{j_0}; \\
  0, & 2^{j_0}|\gamma'(2^{j_{0}-1})|\leq|E_{\lambda}|<2^{j_0};\\
  |E_\lambda|^{{\frac{\alpha-\theta_{2}}{(\omega_{0,1}-\varepsilon)(1+\theta_{2})}}+{\frac{\theta_{2}}{1+\theta_{2}}}}, &|E_{\lambda}|<2^{j_0}|\gamma'(2^{j_{0}-1})|.
\end{cases}$$
The above implies
\begin{align}\label{eq:5.3}
\Upsilon_1\lesssim
\begin{cases}
  \|f\|^{\frac{1}{1+\theta_2}}_{L^{\frac{1}{\theta_2}}(\mathbb{R})}\|g\|^{\frac{1}{1+\theta_2}}_{L^{1}(\mathbb{R})}, & |E_\lambda|\geq2^{j_0}|\gamma'(2^{j_{0}-1})|; \\
  \left(|E_\lambda|^{\frac{\alpha}{1+\theta_{2}}}+|E_\lambda|^{{\frac{\alpha-\theta_{2}}{(\omega_{0,1}-\varepsilon)(1+\theta_{2})}}+{\frac{\theta_{2}}{1+\theta_{2}}}}\right)
\|f\|^{\frac{1}{1+\theta_2}}_{L^{\frac{1}{\theta_2}}(\mathbb{R})}\|g\|^{\frac{1}{1+\theta_2}}_{L^{1}(\mathbb{R})}, &|E_{\lambda}|<2^{j_0}|\gamma'(2^{j_{0}-1})|.
\end{cases}
 \end{align}

For $\Upsilon_2$, we obviously have the estimate
\begin{align}\label{eq:5.4}
\Upsilon_{2}\lesssim
\begin{cases}
   \|f\|^{\frac{1}{1+\theta_2}}_{L^{\frac{1}{\theta_2}}(\mathbb{R})}\|g\|^{\frac{1}{1+\theta_2}}_{L^{1}(\mathbb{R})}, &  |E_{\lambda}|\geq2^{j_0}|\gamma'(2^{j_{0}-1})|; \\
  |E_\lambda|^{\frac{1-\theta_{2}}{1+\theta_{2}}}\|f\|^{\frac{1}{1+\theta_2}}_{L^{\frac{1}{\theta_2}}(\mathbb{R})}\|g\|^{\frac{1}{1+\theta_2}}_{L^{1}(\mathbb{R})}, & |E_{\lambda}|<2^{j_0}|\gamma'(2^{j_{0}-1})|.
 \end{cases}
  \end{align}

For $\Upsilon_3$, which can be written as
$$\left[\sum\limits_{  \substack{j>j_{0}+2   \\ |E_{\lambda}|\geq2^j{|\gamma'(2^{j-1})|} }} 2^{{\frac{\alpha}{1+\theta_{2}}}j}+
\sum\limits_{  \substack{j>j_{0}+2   \\ |E_{\lambda}|<2^j{|\gamma'(2^{j-1})|} }} \frac{2^{{\frac{\alpha-\theta_{2}}{1+\theta_{2}}}j}
|E_\lambda|^{\frac{\theta_{2}}{1+\theta_{2}}}}{|\gamma'(2^{j-1})|^{\frac{\theta_{2}}{1+\theta_{2}}}}\right]
\|f\|^{\frac{1}{1+\theta_2}}_{L^{\frac{1}{\theta_2}}(\mathbb{R})}\|g\|^{\frac{1}{1+\theta_2}}_{L^{1}(\mathbb{R})}.$$
For any $\varepsilon>0$, by Lemma \ref{lemma 2.1}, we can find
$$\sum\limits_{  \substack{j>j_{0}+2   \\ |E_{\lambda}|\geq2^j{|\gamma'(2^{j-1})|} }} 2^{{\frac{\alpha}{1+\theta_{2}}}j}\lesssim
\begin{cases}
  |E_\lambda|^{\frac{\alpha}{(1+\theta_{2})(\omega_{\infty,1}-\varepsilon)}}, & |E_\lambda|\geq2^{j_0}|\gamma'(2^{j_{0}-1})|;\\
  0, & |E_\lambda|<2^{j_0}|\gamma'(2^{j_{0}-1})|.
\end{cases}$$
Moreover, for $\varepsilon\in(0, \omega_{\infty,1}-1]$, by Lemma \ref{lemma 2.1}, then the following estimate is valid:
$$\sum\limits_{  \substack{j>j_{0}+2   \\ |E_{\lambda}|<2^j{|\gamma'(2^{j-1})|} }} \frac{2^{{\frac{\alpha-\theta_{2}}{1+\theta_{2}}}j}
|E_\lambda|^{\frac{\theta_{2}}{1+\theta_{2}}}}{|\gamma'(2^{j-1})|^{\frac{\theta_{2}}{1+\theta_{2}}}}\lesssim
\begin{cases}
  |E_\lambda|^{\frac{\alpha+\theta_2(\omega_{\infty,2}-\omega_{\infty,1}+2\varepsilon)}{(1+\theta_{2})(\omega_{\infty,2}+\varepsilon)}}, & |E_\lambda|\geq2^{j_0}|\gamma'(2^{j_{0}-1})|; \\
  |E_\lambda|^{\frac{\theta_{2}}{1+\theta_{2}}}, & |E_\lambda|<2^{j_0}|\gamma'(2^{j_{0}-1})|.
\end{cases}$$
Because of $|E_\lambda|^{\frac{\alpha}{(1+\theta_{2})(\omega_{\infty,1}-\varepsilon)}}< |E_\lambda|^{\frac{\alpha+\theta_{2}(\omega_{\infty,2}-\omega_{\infty,1}+2\varepsilon)}{(1+\theta_{2})(\omega_{\infty,2}+\varepsilon)}}$ holding for $|E_\lambda|\geq2^{j_0}|\gamma'(2^{j_{0}-1})|$, which comes from the fact that $\alpha<\theta_2(\omega_{\infty,1}-\varepsilon)$ with $\varepsilon\in(0, \omega_{\infty,1}-1]$, we then arrive at
\begin{align}\label{eq:5.5}
\Upsilon_{3}\lesssim
\begin{cases}
  |E_\lambda|^{\frac{\alpha+\theta_{2}(\omega_{\infty,2}-\omega_{\infty,1}+2\varepsilon)}{(1+\theta_{2})(\omega_{\infty,2}+\varepsilon)}}\|f\|^{\frac{1}{1+\theta_2}}_{L^{\frac{1}{\theta_2}}(\mathbb{R})}\|g\|^{\frac{1}{1+\theta_2}}_{L^{1}(\mathbb{R})}, & |E_\lambda|\geq2^{j_0}|\gamma'(2^{j_{0}-1})|; \\
  |E_\lambda|^{\frac{\theta_{2}}{1+\theta_{2}}}\|f\|^{\frac{1}{1+\theta_2}}_{L^{\frac{1}{\theta_2}}(\mathbb{R})}\|g\|^{\frac{1}{1+\theta_2}}_{L^{1}(\mathbb{R})}, & |E_\lambda|<2^{j_0}|\gamma'(2^{j_{0}-1})|.
\end{cases}
 \end{align}

Notice that $\omega_{0,1}>1$, we can set $\varepsilon\in(0, \omega_{0,1}-1]$, and it follows from $\alpha<\theta_2$ that $$|E_\lambda|^{{\frac{\alpha-\theta_{2}}{(\omega_{0,1}-\varepsilon)(1+\theta_{2})}}+{\frac{\theta_{2}}{1+\theta_{2}}}}\leq|E_\lambda|^{\frac{\alpha}{1+\theta_{2}}} ~~\textrm{and} ~~|E_\lambda|^{\frac{\theta_2}{1+\theta_{2}}}<|E_\lambda|^{\frac{\alpha}{1+\theta_{2}}}$$ for $|E_{\lambda}|<2^{j_0}|\gamma'(2^{j_{0}-1})|$. Furthermore, since $\alpha>1-\theta_2$, one can check that $|E_\lambda|^{\frac{\alpha}{1+\theta_{2}}}\leq |E_\lambda|^{\frac{1-\theta_{2}}{1+\theta_{2}}}$ holds for $|E_{\lambda}|<2^{j_0}|\gamma'(2^{j_{0}-1})|$ directly. These estimates, combined with the estimates \eqref{eq:5.2}-\eqref{eq:5.5} and the fact that $ |E_\lambda|^{\frac{\alpha+\theta_{2}(\omega_{\infty,2}-\omega_{\infty,1}+2\varepsilon)}{(1+\theta_{2})(\omega_{\infty,2}+\varepsilon)}}\geq 1$ for $|E_\lambda|\geq2^{j_0}|\gamma'(2^{j_{0}-1})|$, lead to
\begin{align}\label{eq:5.6}
\lambda^{\frac{1}{1+\theta_2}}|E_{\lambda}|\lesssim
\begin{cases}
   |E_\lambda|^{\frac{\alpha+\theta_{2}(\omega_{\infty,2}-\omega_{\infty,1}+2\varepsilon)}{(1+\theta_{2})(\omega_{\infty,2}+\varepsilon)}}\|f\|^{\frac{1}{1+\theta_2}}_{L^{\frac{1}{\theta_2}}(\mathbb{R})}\|g\|^{\frac{1}{1+\theta_2}}_{L^{1}(\mathbb{R})}, & |E_\lambda|\geq2^{j_0}|\gamma'(2^{j_{0}-1})|; \\
  |E_\lambda|^{\frac{1-\theta_{2}}{1+\theta_{2}}}
\|f\|^{\frac{1}{1+\theta_2}}_{L^{\frac{1}{\theta_2}}(\mathbb{R})}\|g\|^{\frac{1}{1+\theta_2}}_{L^{1}(\mathbb{R})}, &|E_{\lambda}|<2^{j_0}|\gamma'(2^{j_{0}-1})|.
\end{cases}
 \end{align}
Therefore, we obtain a similar estimate with \eqref{eq:4.12}. Here, we used \eqref{eq:4.13}. As in Proposition \ref{proposition 4.1}, we conclude that $H_{\alpha,\gamma}$ maps $L^{2r( \varepsilon )}(\mathbb{R})\times L^{1   }(\mathbb{R})$ to $L^{r(\varepsilon ), \infty}(\mathbb{R})$, where $\frac{1}{r( \varepsilon ) }=\frac{2(\omega_{\infty,2}-\alpha)}{2\omega_{\infty,2}-\omega_{\infty,1}}-\varepsilon $ with $\varepsilon$ small enough. Therefore, we finish the proof of Proposition \ref{proposition 5.1}.\end{proof}

\begin{remark}\label{remark 5.2}
Let $\alpha\in(0, \frac{1}{2})$, from the process of the proof of Proposition \ref{proposition 5.1}, we have
\begin{align}\label{eq:5.7}
\left\|H_{\alpha,\gamma}(f,g)\right\|_{L^{\frac{\omega_{\infty,2}}{\omega_{\infty,2}+ (1-\alpha)\omega_{\infty,1}-\alpha}+\varepsilon, \infty}(\mathbb{R})}\lesssim \|f\|_{L^{\frac{1}{1-\alpha}}(\mathbb{R})}\left\|g\right\|_{L^{1}(\mathbb{R})}
\end{align}
for $\varepsilon>0$ small enough.

Indeed, we should only repeat the previous process to prove Proposition \ref{proposition 5.1} except for establishing \eqref{eq:5.6}. Here, we assume that $\alpha\leq1-\theta_2$, this can be realized when $\alpha\in(0, \frac{1}{2})$. We then have that
$|E_\lambda|^{\frac{1-\theta_{2}}{1+\theta_{2}}}\leq |E_\lambda|^{\frac{\alpha}{1+\theta_{2}}}$ holds for $|E_{\lambda}|<2^{j_0}|\gamma'(2^{j_{0}-1})|$, which further implies
\begin{align*}
\lambda^{\frac{1}{1+\theta_2}}|E_{\lambda}|\lesssim
\begin{cases}
   |E_\lambda|^{\frac{\alpha+\theta_{2}(\omega_{\infty,2}-\omega_{\infty,1}+2\varepsilon)}{(1+\theta_{2})(\omega_{\infty,2}+\varepsilon)}}\|f\|^{\frac{1}{1+\theta_2}}_{L^{\frac{1}{\theta_2}}(\mathbb{R})}\|g\|^{\frac{1}{1+\theta_2}}_{L^{1}(\mathbb{R})}, & |E_\lambda|\geq2^{j_0}|\gamma'(2^{j_{0}-1})|; \\
  |E_\lambda|^{\frac{\alpha}{1+\theta_{2}}}
\|f\|^{\frac{1}{1+\theta_2}}_{L^{\frac{1}{\theta_2}}(\mathbb{R})}\|g\|^{\frac{1}{1+\theta_2}}_{L^{1}(\mathbb{R})}, &|E_{\lambda}|<2^{j_0}|\gamma'(2^{j_{0}-1})|.
\end{cases}
 \end{align*}
Notice that $\frac{\alpha+\theta_{2}(\omega_{\infty,2}-\omega_{\infty,1}+2\varepsilon)}{(1+\theta_{2})(\omega_{\infty,2}+\varepsilon)}\leq\frac{\alpha}{1+\theta_{2}}$ is equivalent to $\theta_2\leq\frac{\alpha(\omega_{\infty,2}-1+\varepsilon)}{\omega_{\infty,2}-\omega_{\infty,1}+2\varepsilon}$. This can be established by the fact that $\alpha<\frac{\alpha(\omega_{\infty,2}-1+\varepsilon)}{\omega_{\infty,2}-\omega_{\infty,1}+2\varepsilon} $ with $\varepsilon\in(0, \omega_{\infty,1}-1)$. Thus
\begin{align*}
\lambda|E_{\lambda}|^{\frac{\omega_{\infty,2}+ \theta_2(\omega_{\infty,1}-\varepsilon)-\alpha+\varepsilon}{\omega_{\infty,2}+\varepsilon}}\lesssim \|f\|_{L^{\frac{1}{\theta_2}}(\mathbb{R})}\|g\|_{L^{1}(\mathbb{R})}
\end{align*}
holds for all $0<|E_\lambda|<\infty$. On the other hand, since $\frac{\omega_{\infty,2}-\omega_{\infty,1}}{\omega_{\infty,2}-1}<\frac{\alpha}{1-\alpha}$, we have $1-\alpha\leq \frac{\alpha(\omega_{\infty,2}-1+\varepsilon)}{\omega_{\infty,2}-\omega_{\infty,1}+2\varepsilon}$ for $\varepsilon>0$ small enough.\footnote{We note that $\frac{\alpha}{1-\alpha}<\frac{\omega_{\infty,2}}{\omega_{\infty,2}-\omega_{\infty,1}+1}$ holds for all $\alpha\in(0, \frac{1}{2})$ since $\omega_{\infty,1}>1$.} Therefore, we can take $\theta_2=1-\alpha$. Observe that $$\frac{\omega_{\infty,2}+ (1-\alpha)(\omega_{\infty,1}-\varepsilon)-\alpha+\varepsilon}{\omega_{\infty,2}+\varepsilon}<\frac{\omega_{\infty,2}+ (1-\alpha)\omega_{\infty,1}-\alpha} {\omega_{\infty,2}}$$ for $\varepsilon>0$ small enough, this gives the assertion \eqref{eq:5.7}.
\end{remark}

\begin{proposition}\label{proposition 5.2}
Let $\gamma $ be the same as in Theorem \ref{maintheorem2}. Then, for $\varepsilon>0$ small enough, there exists a positive constant $C$, such that
$$\left\|H_{\alpha,\gamma}(f,g)\right\|_{L^{\tilde{r}(\varepsilon),\infty}(\mathbb{R})}\leq C \|f\|_{L^{2\tilde{r}(\varepsilon)}(\mathbb{R})}\left\|g\right\|_{L^{1}(\mathbb{R})},$$
where $\frac{1}{\tilde{r}( \varepsilon ) }:=\frac{2(\omega_{\infty,1}-\alpha)}{\omega_{\infty,1}}-\varepsilon $ with $\varepsilon$ small enough.
\end{proposition}

\begin{proof}

Repeating the proof of Proposition \ref{proposition 5.1}, we then arrive at \eqref{eq:5.2}. Consider $\Upsilon_{1}$, we now assume that $\alpha>\theta_2$. For any $\varepsilon>0$, by Lemma \ref{lemma 2.1}, one sees
$$\sum\limits_{  \substack{j<j_{0}   \\ 2^j{|\gamma'(2^{j-1})|}\leq|E_{\lambda}|<2^j }} 2^{{\frac{\alpha-\theta_{2}}{1+\theta_{2}}}j}|E_\lambda|^{\frac{\theta_{2}}{1+\theta_{2}}}\lesssim
\begin{cases}
  0, & |E_\lambda|\geq2^{j_0}; \\
  |E_\lambda|^{\frac{\alpha+\theta_2(\omega_{0,2}-1+\varepsilon)}{(\omega_{0,2}+\varepsilon)(1+\theta_{2})}}, & 2^{j_0}|\gamma'(2^{j_{0}-1})|\leq|E_{\lambda}|<2^{j_0};\\
  |E_\lambda|^{\frac{\alpha+\theta_2(\omega_{0,2}-1+\varepsilon)}{(\omega_{0,2}+\varepsilon)(1+\theta_{2})}}, &|E_{\lambda}|<2^{j_0}|\gamma'(2^{j_{0}-1})|.
\end{cases}$$
Notice that $\alpha-\omega_{0,2}-\varepsilon+\theta_2(\omega_{0,2}-1+\varepsilon)<0$, this along with Lemma \ref{lemma 2.1} yields
$$\sum\limits_{  \substack{j<j_{0}   \\ |E_{\lambda}|<2^j|\gamma'(2^{j-1})| }} \frac{2^{{\frac{\alpha-1}{1+\theta_{2}}j}}|E_\lambda|^{\frac{1}{1+\theta_{2}}}}{|\gamma'(2^{j-1})|^{\frac{1-\theta_{2}}{1+\theta_{2}}}}\lesssim
\begin{cases}
  0, & |E_\lambda|\geq2^{j_0}; \\
  0, & 2^{j_0}|\gamma'(2^{j_{0}-1})|\leq|E_{\lambda}|<2^{j_0};\\
  |E_\lambda|^{\frac{\alpha+\omega_{0,1}-\omega_{0,2}-2\varepsilon+\theta_{2}(\omega_{0,2}-1+\varepsilon)}{(\omega_{0,1}-\varepsilon)(1+\theta_{2})}}, &|E_{\lambda}|<2^{j_0}|\gamma'(2^{j_{0}-1})|.
\end{cases}$$
It is easy to check that $$|E_\lambda|^{{\frac{\alpha}{1+\theta_{2}}}}\leq|E_\lambda|^{\frac{\alpha+\theta_2(\omega_{0,2}-1+\varepsilon)}{(\omega_{0,2}+\varepsilon)(1+\theta_{2})}}\leq
|E_\lambda|^{\frac{\alpha+\omega_{0,1}-\omega_{0,2}-2\varepsilon+\theta_{2}(\omega_{0,2}-1+\varepsilon)}{(\omega_{0,1}-\varepsilon)(1+\theta_{2})}}$$
holds for $|E_{\lambda}|<2^{j_0}|\gamma'(2^{j_{0}-1})|$. We then have
\begin{align}\label{eq:5.8}
\Upsilon_1\lesssim
\begin{cases}
  \|f\|^{\frac{1}{1+\theta_2}}_{L^{\frac{1}{\theta_2}}(\mathbb{R})}\|g\|^{\frac{1}{1+\theta_2}}_{L^{1}(\mathbb{R})}, & |E_\lambda|\geq2^{j_0}|\gamma'(2^{j_{0}-1})|;\\
  |E_\lambda|^{\frac{\alpha+\omega_{0,1}-\omega_{0,2}-2\varepsilon+\theta_{2}(\omega_{0,2}-1+\varepsilon)}{(\omega_{0,1}-\varepsilon)(1+\theta_{2})}}
\|f\|^{\frac{1}{1+\theta_2}}_{L^{\frac{1}{\theta_2}}(\mathbb{R})}\|g\|^{\frac{1}{1+\theta_2}}_{L^{1}(\mathbb{R})}, &|E_{\lambda}|<2^{j_0}|\gamma'(2^{j_{0}-1})|.
\end{cases}
 \end{align}

$\Upsilon_2$ has the same estimate as in Proposition \ref{proposition 5.1}, and we turn to $\Upsilon_3$. It is different from \eqref{eq:5.b}, we here need to bound \eqref{eq:5.a} as
\begin{align}\label{eq:5.000}
  \left\|H_{\alpha,\gamma,j}(f,g)\chi_{E}\right\|_{L^{\frac{1}{1+\theta_2}}(\mathbb{R})}
\leq \min\left\{2^j,  \frac{2^{j\theta_2}|E|^{1-\theta_2}}{|\gamma'(2^{j-1})|^{1-\theta_2}} \right\} \|f\|_{L^{\frac{1}{\theta_2}}(\mathbb{R})}\|g\|_{L^{1}(\mathbb{R})}.
 \end{align}
Then, we can write
\begin{align*}
\Upsilon_{3}=\sum\limits_{j>j_{0}+2}2^{{\frac{\alpha-1}{1+\theta_{2}}}j}\min\left\{2^j,  \frac{2^{j\theta_2}|E_{\lambda}|^{1-\theta_2}}{|\gamma'(2^{j-1})|^{1-\theta_2}} \right\}^{\frac{1}{1+\theta_2}} \|f\|^{\frac{1}{1+\theta_2}}_{L^{\frac{1}{\theta_2}}(\mathbb{R})}\|g\|^{\frac{1}{1+\theta_2}}_{L^{1}(\mathbb{R})},
 \end{align*}
which can be rewritten as
$$\left[\sum\limits_{  \substack{j>j_{0}+2   \\ |E_{\lambda}|\geq2^j{|\gamma'(2^{j-1})|} }} 2^{{\frac{\alpha}{1+\theta_{2}}}j}+
\sum\limits_{  \substack{j>j_{0}+2   \\ |E_{\lambda}|<2^j{|\gamma'(2^{j-1})|} }} \frac{2^{{\frac{\alpha-1+\theta_{2}}{1+\theta_{2}}}j}
|E_\lambda|^{\frac{1-\theta_{2}}{1+\theta_{2}}}}{|\gamma'(2^{j-1})|^{\frac{1-\theta_{2}}{1+\theta_{2}}}}\right]
\|f\|^{\frac{1}{1+\theta_2}}_{L^{\frac{1}{\theta_2}}(\mathbb{R})}\|g\|^{\frac{1}{1+\theta_2}}_{L^{1}(\mathbb{R})}.$$
For any $\varepsilon>0$, by Lemma \ref{lemma 2.1}, we also have
$$\sum\limits_{  \substack{j>j_{0}+2   \\ |E_{\lambda}|\geq2^j{|\gamma'(2^{j-1})|} }} 2^{{\frac{\alpha}{1+\theta_{2}}}j}\lesssim
\begin{cases}
  |E_\lambda|^{\frac{\alpha}{(1+\theta_{2})(\omega_{\infty,1}-\varepsilon)}}, & |E_\lambda|\geq2^{j_0}|\gamma'(2^{j_{0}-1})|;\\
  0, & |E_\lambda|<2^{j_0}|\gamma'(2^{j_{0}-1})|.
\end{cases}$$
We assume that $\alpha<(1-\theta_2)(\omega_{\infty,1}-\varepsilon)$ with $\varepsilon\in(0, \omega_{\infty,1})$, it follows from Lemma \ref{lemma 2.1} that
$$\sum\limits_{  \substack{j>j_{0}+2   \\ |E_{\lambda}|<2^j{|\gamma'(2^{j-1})|} }} \frac{2^{{\frac{\alpha-1+\theta_{2}}{1+\theta_{2}}}j}
|E_\lambda|^{\frac{1-\theta_{2}}{1+\theta_{2}}}}{|\gamma'(2^{j-1})|^{\frac{1-\theta_{2}}{1+\theta_{2}}}}\lesssim
\begin{cases}
  |E_\lambda|^{\frac{\alpha+(1-\theta_2)(\omega_{\infty,2}-\omega_{\infty,1}+2\varepsilon)}{(1+\theta_{2})(\omega_{\infty,2}+\varepsilon)}}, & |E_\lambda|\geq2^{j_0}|\gamma'(2^{j_{0}-1})|; \\
  |E_\lambda|^{\frac{1-\theta_{2}}{1+\theta_{2}}}, & |E_\lambda|<2^{j_0}|\gamma'(2^{j_{0}-1})|.
\end{cases}$$
Based on the fact $|E_\lambda|^{\frac{\alpha}{(1+\theta_{2})(\omega_{\infty,1}-\varepsilon)}}< |E_\lambda|^{\frac{\alpha+(1-\theta_{2})(\omega_{\infty,2}-\omega_{\infty,1}+2\varepsilon)}{(1+\theta_{2})(\omega_{\infty,2}+\varepsilon)}}$ holds for $|E_\lambda|\geq2^{j_0}|\gamma'(2^{j_{0}-1})|$, we deduce that
\begin{align}\label{eq:5.11}
\Upsilon_{3}\lesssim
\begin{cases}
  |E_\lambda|^{\frac{\alpha+(1-\theta_{2})(\omega_{\infty,2}-\omega_{\infty,1}+2\varepsilon)}{(1+\theta_{2})(\omega_{\infty,2}+\varepsilon)}}\|f\|^{\frac{1}{1+\theta_2}}_{L^{\frac{1}{\theta_2}}(\mathbb{R})}\|g\|^{\frac{1}{1+\theta_2}}_{L^{1}(\mathbb{R})}, & |E_\lambda|\geq2^{j_0}|\gamma'(2^{j_{0}-1})|; \\
  |E_\lambda|^{\frac{1-\theta_{2}}{1+\theta_{2}}}\|f\|^{\frac{1}{1+\theta_2}}_{L^{\frac{1}{\theta_2}}(\mathbb{R})}\|g\|^{\frac{1}{1+\theta_2}}_{L^{1}(\mathbb{R})}, & |E_\lambda|<2^{j_0}|\gamma'(2^{j_{0}-1})|.
\end{cases}
 \end{align}

Combining \eqref{eq:5.8}, \eqref{eq:5.4} and \eqref{eq:5.11}, then the following estimate is valid:
\begin{align*}
\lambda^{\frac{1}{1+\theta_2}}|E_{\lambda}|\lesssim
\begin{cases}
   |E_\lambda|^{\frac{\alpha+(1-\theta_{2})(\omega_{\infty,2}-\omega_{\infty,1}+2\varepsilon)}{(1+\theta_{2})(\omega_{\infty,2}+\varepsilon)}}\|f\|^{\frac{1}{1+\theta_2}}_{L^{\frac{1}{\theta_2}}(\mathbb{R})}\|g\|^{\frac{1}{1+\theta_2}}_{L^{1}(\mathbb{R})}, & |E_\lambda|\geq2^{j_0}|\gamma'(2^{j_{0}-1})|;\\
\left(|E_\lambda|^{\frac{\alpha+\omega_{0,1}-\omega_{0,2}-2\varepsilon+\theta_{2}(\omega_{0,2}-1+\varepsilon)}{(\omega_{0,1}-\varepsilon)(1+\theta_{2})}}+|E_\lambda|^{\frac{1-\theta_2}{1+\theta_2}}\right)
\|f\|^{\frac{1}{1+\theta_2}}_{L^{\frac{1}{\theta_2}}(\mathbb{R})}\|g\|^{\frac{1}{1+\theta_2}}_{L^{1}(\mathbb{R})}, &|E_{\lambda}|<2^{j_0}|\gamma'(2^{j_{0}-1})|.
\end{cases}
\end{align*}
If we have $$\frac{\alpha+(1-\theta_{2})(\omega_{\infty,2}-\omega_{\infty,1}+2\varepsilon)}{(1+\theta_{2})(\omega_{\infty,2}+\varepsilon)}<
\frac{1-\theta_2}{1+\theta_2}\leq
\frac{\alpha+\omega_{0,1}-\omega_{0,2}-2\varepsilon+\theta_{2}(\omega_{0,2}-1+\varepsilon)}{(\omega_{0,1}-\varepsilon)(1+\theta_{2})},$$ which equals to
$\frac{\omega_{0,2}-\alpha+\varepsilon}{\omega_{0,2}+\omega_{0,1}-1}\leq\theta_2<\frac{\omega_{\infty,1}-\alpha-\varepsilon}{\omega_{\infty,1}-\varepsilon}$ for $\varepsilon>0$ small enough, we then obtain that
\begin{align}\label{eq:5.12}
\lambda|E_{\lambda}|^{\frac{\omega_{\infty,1}-\alpha-\varepsilon+ \theta_2(2\omega_{\infty,2}- \omega_{\infty,1} +3\varepsilon)}{\omega_{\infty,2}+\varepsilon}}\lesssim \|f\|_{L^{\frac{1}{\theta_2}}(\mathbb{R})}\|g\|_{L^{1}(\mathbb{R})}
\end{align}
holds for all $0<|E_\lambda|<\infty$. Indeed, $\theta_2<\frac{\omega_{\infty,1}-\alpha-\varepsilon}{\omega_{\infty,1}-\varepsilon}$ comes from the assumption in estimating $\Upsilon_{3}$. In order to show that we can take $\theta_2\geq\frac{\omega_{0,2}-\alpha+\varepsilon}{\omega_{0,2}+\omega_{0,1}-1}$, it suffices to show that $$\frac{\omega_{0,2}-\alpha+\varepsilon}{\omega_{0,2}+\omega_{0,1}-1}<\frac{\omega_{\infty,1}-\alpha-\varepsilon}{\omega_{\infty,1}-\varepsilon}\leq \alpha $$ holds for $\varepsilon>0$ small enough, which is trivially verified by the condition that $\frac{\alpha}{1-\alpha}\geq \omega_{\infty,1}\geq \omega_{0,2}$ and $\omega_{0,1}>1$. Thus, we obtain \eqref{eq:5.12} for $\theta_2<\frac{\omega_{\infty,1}-\alpha-\varepsilon}{\omega_{\infty,1}-\varepsilon}$ with $\varepsilon>0$ small enough.

Notice that $$\frac{\omega_{\infty,1}-\alpha-\varepsilon+ \theta_2(2\omega_{\infty,2}- \omega_{\infty,1} +3\varepsilon)}{\omega_{\infty,2}+\varepsilon}<\frac{2(\omega_{\infty,1}-\alpha-\varepsilon)}{\omega_{\infty,1}-\varepsilon}$$ when $\theta_2<\frac{\omega_{\infty,1}-\alpha-\varepsilon}{\omega_{\infty,1}-\varepsilon}$ with $\varepsilon>0$ small enough, and $\frac{2(\omega_{\infty,1}-\alpha-\varepsilon)}{\omega_{\infty,1}-\varepsilon}<\frac{2(\omega_{\infty,1}-\alpha)}{\omega_{\infty,1}}$, we conclude that $H_{\alpha,\gamma}$ maps $L^{2\tilde{r}( \varepsilon )}(\mathbb{R})\times L^{1   }(\mathbb{R})$ to $L^{\tilde{r}(\varepsilon ), \infty}(\mathbb{R})$, where $\frac{1}{\tilde{r}( \varepsilon ) }=\frac{2(\omega_{\infty,1}-\alpha)}{\omega_{\infty,1}}-\varepsilon $ with $\varepsilon$ small enough. Therefore, we finish the proof of Proposition \ref{proposition 5.2}.
\end{proof}

\bigskip

\noindent  Junfeng Li

\smallskip

\noindent  School of Mathematical Sciences, Dalian University of Technology, Dalian, 116024,  People's Republic of China

\smallskip

\noindent {\it E-mail}: \texttt{junfengli@dlut.edu.cn}

\bigskip

\noindent Haixia Yu (Corresponding author) and Minqun Zhao

\smallskip

\noindent  Department of Mathematics, Shantou University, Shantou, 515821, People's Republic of China

\smallskip

\noindent {\it E-mails}: \texttt{hxyu@stu.edu.cn} (H. Yu)

\noindent\phantom{{\it E-mails:}} \texttt{22mqzhao@stu.edu.cn} (M. Zhao)

\end{document}